\documentclass[11pt]{article}
\usepackage{amsmath}
\usepackage{ulem}
  \usepackage{graphics}
  \usepackage{amsfonts}
  \usepackage{epsfig}
  \usepackage{float}
  \usepackage{graphicx}
 \usepackage{epstopdf}
 \usepackage{amssymb}
 \usepackage{mathrsfs}
 \usepackage{amsthm}
 \usepackage[colorlinks=true]{hyperref}
 \usepackage[numbers]{natbib}
\usepackage{color}
\usepackage{dsfont}
\hypersetup{urlcolor=blue, citecolor=red}
\usepackage{hyperref}
\hypersetup{%
  colorlinks = true,
  linkcolor  = black
}

\numberwithin{equation}{section}

\newtheorem{theorem}{Theorem}[section]
\newtheorem{corollary}{Corollary}[section]
\newtheorem{lemma}[theorem]{Lemma}

\newtheorem{proposition}{Proposition}[section]
\newtheorem{conjecture}[theorem]{Conjecture}

\newtheorem{definition}[theorem]{Definition}

\newtheorem{remark}[theorem]{Remark}

\def\bar{\overline}

\DeclareMathOperator\MF{{\mathcal{MF}}}
\DeclareMathOperator\PMF{{\mathcal{PMF}}}
\DeclareMathOperator\M{Mod}
\DeclareMathOperator\T{Teich}
\DeclareMathOperator\E{Ext}

\begin{document}

\title{Boundary representations of mapping class groups}
\author{Biao Ma}

\date{}
\maketitle
\begin{abstract}
Let $S = S_g$ be a closed orientable surface of genus $g \geq 2$ and $\M(S)$ be the mapping class group of $S$. In this paper, we show that the boundary representation of $\M(S)$ is ergodic using statistical hyperbolicity, which generalizes the classical result of Masur on ergodicity of the action of $\M(S)$ on the projective measured foliation space $\PMF(S).$ As a corollary, we show that the boundary representation of $\M(S)$ is irreducible.
\end{abstract}
\tableofcontents
\section{Introduction}
 Let $S = S_{g}$ be a closed, connected, orientable surface of genus $g$. Recall that the mapping class group $\M(S)$ of $S$ is defined to be the group of isotopy classes of orientation-preserving homeomorphisms of $S$. Throughout this paper, the genus $g$ is assumed to be at least $2$. {\it The space of measured foliations} $\MF(S)$ is the set of equivalence classes of non-zero measured foliations on $S$. The mapping class group $\M(S)$ acts on $\MF(S)$ and preserves a Radon measure $\nu$, called the {\it Thurston measure} on $\MF(S)$. Moreover, the space $\MF(S)$ is equipped with an $\mathbb{R}_{+}-$action that commutes with the $\M(S)-$action. Therefore, $\M(S)$ acts on the quotient $\PMF(S)$, called {\it the projective measured foliation space}, of $\MF(S)$ by $\mathbb{R}_{+}$ preserving the measure class $[\nu]$, called the Thurston measure (class) on $\PMF(S)$, defined by the Thurston measure on $\MF(S)$.\\

 One motivation of this paper is to use geometric objects, such as $\MF(S)$ and $\PMF(S)$ to understand unitary representations of $\M(S)$ (see, for instance, \cite{Paris2002},\cite{Costantino-Martelli14} for related topics).\\

\noindent \textbf{Main results.}~~~Recall that, for a probability measure class preserving action of $G$ on $(X,[\nu])$, one defines a unitary representation of $G$ on $L^2(X,\nu)$, called a {\it quasi-regular representation} (see Section \ref{subsection:quasi-regularrepresentation} for definitions). Hence, for a probability measure class preserving ergodic action, it is natural to ask that whether the quasi-regular representation is irreducible. Recall that a unitary representation is called \textit{irreducible} if it has no nontrivial closed invariant subspaces. Notice that this is not true for a measure preserving ergodic action as it always has $\mathbb{C}\mathds{1}_{X}$ as a nontrivial closed invariant subspace. For the ergodic action of $\M(S)$ on $\PMF(S)$ with respect to $[\nu]$, we prove:
\begin{theorem}[See Corollary \ref{corollary:mappingclassgroupirreducibility}]
	Let $S = S_g$ be a closed surface of genus $g \geq 2$. The quasi-regular unitary representation of the mapping class group $\M(S)$ on $L^2(\PMF(S),\nu)$, the space of square integrable functions on $\PMF(S)$ with respect to the Thurston measure $\nu$, is irreducible.
\end{theorem}

This theorem is a corollary of an ergodic-type theorem for quasi-regular representations and it is this ergodic-type theorem that is much interesting and has many corollaries.\\

Given an action of a discrete group $G$ on a Borel probabililty space $(X,\mu)$ preserving the measure class $[\mu]$, denote the associated quasi-regular unitary representation $G \curvearrowright L^2(X,\mu)$ by $\pi_{\mu}$ and the projection to the subspace of constant functions by $P_{\mathds{1}_X}$. The representation $\pi_{\mu}$ is said to be \textit{ergodic with respect to $(K_n,e_n,f)$}, where $K_n \subset G$ is a finite subset with $|K_n| \to \infty$, $e_n: K_n \to X$ is a map and $f$ is a bounded Borel function on $X$, if we have the following convergence in weak operator topology:
$$\frac{1}{|K_n|} \sum_{g \in K_n} f(e_n(g))\frac{\pi_{\mu}(g)}{\langle \pi_{\mu}(g)\mathds{1}_X,\mathds{1}_X\rangle} \to fP_{\mathds{1}_X}.$$

The ergodicity of representations generalizes the ergodicity of group actions and it in fact implies that the irreducibility \cite[Proposition 2.5]{BLP}. Our main result is

\begin{theorem}[See Theorem \ref{theoreom:ergodicitymcg}]\label{intro:maintheorem}

There exist a sequence of finite subsets $\{E_n\}$ of $\M(S)$ and maps $e_n: E_n \to \PMF(S)$ such that, for every bounded Borel functions $f$, the quasi-regular representation $\pi_{\nu}$, with respect to the Thurston measure $\nu$, is ergodic with respect to $(E_n,e_n,f)$. 
\end{theorem}
Thus we obtain a generalization of Masur's classical result on ergodicity of the action $\M(S) \curvearrowright \PMF(S)$ \cite{Masur_ergodic}. One of the key steps in our proof is the following result concerning matrix coefficients of representations which might be of independent interests.
\begin{theorem}[See Theorem \ref{Harish-Chandra} for the precise statement]
	 There exist a sequence of finite subsets $\{ E_n\}$ of $\M(S)$ for $n >> 0$ with exponential growth and constants $a_1 > 0, a_2 > 0,b_1, b_2,c_1 >0$ such that for every $g \in E_n,$ $$ 
	(a_1n-c_1\ln \ln n +b_1)e^{-\frac{h}{2}n} \leq \langle \pi_{\nu}(g)\mathds{1}_{\PMF(S)},\mathds{1}_{\PMF(S)}\rangle \leq (a_2n+b_2)e^{-\frac{h}{2}n}. $$
\end{theorem}

The above theorem should be compared with \cite[Proposition 3.2]{Boyer17}. Finally, we remark that, for this quasi-regular representation of $\M(S)$, we can also show that it is \textit{tempered}, meaning that it is weakly contained in the regular representation of $\ell^2(\M(S))$ (see Proposition \ref{proposition:tempered}).\\

\noindent \textbf{Historical remarks.}~~~The main theorem is related to a question of Bader-Muchnik in the context of random walks on groups. Namely, let $G$ be a discrete group and $\mu$ be a probability measure on $G$. Let $(\partial G,\nu)$ be the Poisson boundary of $G$ associated to the $\mu-$random walk on $G$. Then the measure class $[\nu]$ is $G-$invariant, hence defines a quasi-regular representation of $G$ on $L^2(\partial G, \nu)$. In \cite{BaderMuchnik}, inspired by the cases of free groups and lattices in Lie groups, Bader-Muchnik proposed the following conjecture.
\begin{conjecture}[\cite{BaderMuchnik}]
	For a locally compact group $G$ and a spread-out probability measure $\mu$ on $G$, the quasi-regular representation associated to the $\mu-$Poisson boundary of $G$ is irreducible.
\end{conjecture}
 We now mention briefly some progress toward this conjecture. As mentioned above, this conjecture is true for certain random walks on free groups and lattices in Lie groups (see \cite{BaderMuchnik} and references therein). Hence it is true for the mapping class group $\M(S) = SL(2,\mathbb{Z})$ of closed surface of genus one acting on $\PMF(S) = S^1$ with respect to the Lebesgue measure which is also identified with the Thurston measure on $\PMF(S)$. Notice that all identifications are $\M(S)-$equvariant. For lattices in Lie groups, one can also deduce the irreducibility from ergodicity of the associated quasi-regular representation (see \cite{BLP}). The conjecture is then verified in \cite{BaderMuchnik} for the fundamental group of compact negatively curved manifolds with respect to the Patterson-Sullivan measure by Bader-Muchnik. Their result has been further generalized to hyperbolic groups \cite{garncarek2016boundary} with respect to the Patterson-Sullivan measure by Garncarek and some discrete subgroups of the group of isometries of a $CAT(-1)$ space with non-arithmetic spectrum by Boyer \cite{Boyer17}. Note that in all cases above,  the Patterson-Sullivan measure on the Gromov boundary coincides with the Poisson boundary  of $(G,\mu)$ for some probability measure $\mu$ on $G$. However, Bj\"{o}rklund-Hartman-Oppelmayer \cite{bjorklund2020random} recently showed that there are random walks on some Lamplighter groups and solvable Baumslag-Solitar groups that provide counterexamples to this conjecture.\\

The relationship between the main theorem and above progress is the following. On the one hand, there is a long history on exploiting similarities between mapping class groups and hyperbolic groups which is quite fruitful. To name very few among massive literatures, we mention \cite{MasurWolf}, \cite{MasurMinsky99} and \cite{Hamenstaedt2009a}. On the other hand, by \cite{ABEM}, the Thurston measure on $\PMF(S)$ is the Patterson-Sullivan measure on the Teichm\"{u}ller boundary (more precisely, Gardiner-Masur boundary, but this is irrelevant for our purpose) of the Teichm\"{u}ller space of $S$ which is in the similar situation with the previous known cases. \\

\noindent {\bf Strategy of the  proof.} The proof of Theorem \ref{intro:maintheorem} exhibits both homogeneous and hyperbolic features of Teichm\"{u}ller spaces. On the one hand, by regarding the Teichm\"{u}ller space $\T(S)$ of $S$ as the homogeneous space of $\M(S)$ and $\PMF(S)$ as the boundary, we follow the approach in Boyer-Pittet-Link \cite{BLP}, namely Theorem \ref{criterion2}, which works quite well for lattices in Lie groups. However, Teichm\"{u}ller spaces are in general not homogeneous spaces, so on the other hand, we make heavily use of hyperbolic features of Teichm\"{u}ller spaces, namely the statistical hyperbolicity defined in Dowdall-Duchin-Masur \cite{DDM}. There are three main steps in the whole proof. The first step is to construct desired finite subsets of $\M(S)$. This is achieved by carefully choosing elements in $\M(S)$ with enough hyperbolicity so that the cardinality of these subsets goes to infinity (in fact, we need the growth to be exponential). The subsets are described before Lemma \ref{lemma:longsegment} relying on \cite{DDM}. The second step is the main part of this paper, that is, the Harish-Chandra estimates (Theorem \ref{Harish-Chandra}). Unlike what have been done in \cite{BaderMuchnik} and \cite{Boyer17}, we deduce the Harish-Chandra estimates using Teichm\"{u}ller theory, especially extremal lengths and intersection number functions. The idea is that, instead of doing estimations directly, we first relate it to integrations on intersection numbers and then use the map considered in Masur-Minsky \cite{MasurMinsky99} which relates $\T(S)$ to the pants curves of $S$ to simplify integrations. It is one of the novelty of this paper and obtains condition (3) in Theorem \ref{criterion2} for free. The last step is to obtain uniform boundness for operators. We use statistical hyperbolicity here as well. As we don't have a metric structure on $\PMF(S)$ with nice regularities, we make use of again intersection numbers and unlike \cite{Boyer17}, we complete the proof by counting lattice points in balls. Finally, we remark that the idea of using intersection number functions as a kind of conformal metrics actually has already been used by Rees in \cite{Rees}.  

\subsection*{Acknowledgments.}
This paper is part of the author's thesis. The author would like to thank his advisor Professor Indira Chatterji for many discussions and the LJAD at Nice for its hospitality. He is also grateful to Professors Adrien Boyer, Ilya Gekhtman, Steven Kerckhoff, Wenyuan Yang for helpful discussions. Part of the work was done while the author participated in the program \textit{Random and Arithmetic Structures in Topology} hosted by the Mathematical Sciences Research Institute in Berkeley, California, during the Fall 2020 semester. This paper is supported by China Scholarship Council (No. 201706140166) and ISF Grant No. GIF I-1485-304.6/2019.

\section{Quasi-regular unitary representations}
\subsection{Quasi-regular representations of discrete groups.}\label{subsection:quasi-regularrepresentation}
In this section, we will recall ergodic quasi-regular representations and a criterion for showing ergodicity of representations. The reader is referred to \cite{BaderMuchnik},\cite{Boyer17} and \cite{BLP} for more details. \\

\noindent{\bf Quasi-regular unitary representations.}~~~ Let $G$ be a locally compact second-countable group and $X$ be a second-countable Hausdorff topological space. Let $\nu$ be a probability Borel measure on $X$. Assume that $G$ acts on $X$ as homeomorphisms and $G$ preserves the measure class $[\nu]$ of $\nu$, namely, $G$ preserves null $\nu-$measure sets . Choose $\nu \in [\nu]$, thus for every $\gamma \in G$, the measure $\gamma_{*}\nu$ is absolutely continuous with respect to $\nu$ and $\nu$ is absolutely continuous with respect to $\gamma_{*}\nu$. Denote the corresponding Radon-Nikodym derivative by $c(\gamma,\nu) = \frac{d\gamma_{*}\nu}{d\nu}$. One can construct a unitary representation $\pi_{\nu}$ of $G$ on $L^2(X,\nu)$ as follows. For every $f \in L^2(X,\nu)$, every $x \in X$ and every $\gamma \in G $, $\pi_{\nu}(\gamma)f(x)$ is defined to be $\pi_{\nu}(\gamma)f(x) = f(\gamma^{-1}x)c(\gamma, \nu)^{\frac{1}{2}}(x).$ The representation $\pi_{\nu}$ will be called a {\it quasi-regular (unitary) representation} of $G$. We remark that if $\nu$ and $\mu$ are in the same measure class, then $\pi_{\nu}$ and $\pi_{\mu}$ are unitary equivalent. Assume that $c(\gamma,\nu)^{\frac{1}{2}}$ is integrable for each $\gamma \in G$ with respect to $\nu$. The {\it Harish-Chandra function} $\Phi$ associated to $\pi_{\nu}$ is then defined to be the integral $$\Phi(\gamma) = \langle \pi_{\nu}(\gamma)\mathds{1}_{X}, \mathds{1}_{X}\rangle_{L^2(X,\nu)} =  \int_{X}c(\gamma,\nu)^{\frac{1}{2}}(x)d\nu(x).$$

\noindent{\bf Ergodic quasi-regular representations.}~~~ From now on, we always assume that $G$ is a discrete group. Let $(X,\nu),\pi_{\nu}$ as above and $\mathcal{B}(L^2(X,\nu))$ be the Banach space of bounded operators on $L^2(X,\nu)$. Let $e_{K}: K \longrightarrow X$ be a map from a finite subset $K$ of $G$ to $X$ and $f:X \longrightarrow \mathbb{C}$ be a bounded Borel function. Consider the following elements in $\mathcal{B}(L^2(X,\nu))$:
\begin{displaymath}
\begin{aligned}
& M^{f}_{(K,e_{K})}: L^2(X,\nu) \longrightarrow L^2(X,\nu),  \phi  \mapsto \frac{1}{|K|}\sum_{\gamma \in K}f(e_{K}(\gamma))\frac{\pi_{\nu}(\gamma)\phi}{\Phi(\gamma)},\\
&\phantom{=\;} P_{\mathds{1}_{X}}: L^2(X,\nu) \longrightarrow L^2(X,\nu), \phi \mapsto \int_{X}\phi d\nu\mathds{1}_{X},\\
&\phantom{=\;} m(f): L^2(X,\nu) \longrightarrow L^2(X,\nu), \phi \mapsto f\phi.
\end{aligned}
\end{displaymath}
We now introduce an ergodicity for quasi-regular representations which generalizes the usual ergodicity for measure class-preserving group actions. Recall that a sequence $F_n \in \mathcal{B}(L^2(X,\nu))$ {\it converges} to $F \in \mathcal{B}(L^2(X,\nu))$, written as $F_n \rightarrow F$, in the weak operator topology if, for every $\phi, \psi \in L^2(X,\nu)$, $\lim_{n \rightarrow \infty}\langle F_n(\phi),\psi \rangle_{L^2} = \langle F(\phi),\psi \rangle_{L^2}$.
\begin{definition}[\cite{BLP}]\label{definition:ergodicityofrepresentation}
	Let $G, (X,\nu),\pi_{\nu},f$ as above. Suppose that for every $n \in \mathbb{N}$, there is a pair $(K_n, e_n: K_n \longrightarrow X)$ such that $K_n$ is a finite subset of $G$ and such that $|K_n| \rightarrow \infty$ as $n \rightarrow \infty$. The representation $\pi_{\nu}$ is called ergodic with respect to $(K_n, e_n)$ and $f$, if we have the following convergence in the weak operator topology: $$M^{f}_{(K_n,e_{n})} \rightarrow m(f)P_{\mathds{1}_{X}}.$$
\end{definition}
\begin{remark}\label{remark:quasiregularirreducible}
	It is easy to see that the ergodicity of a measure class-preserving group action is weaker than the ergodicity of the associated quasi-regular representation. One could refer to [\cite{BLP}, Proposition 2.5] for its proof. 
\end{remark}
The following criterion for the ergodicity of a quasi-regular representation is essentially contained in \cite{BaderMuchnik} and summarized in \cite{BLP}.

\begin{theorem}[\cite{BLP} Theorem 2.2]\label{criterion}
	Let $G, (X, \nu)$ as above and $\pi_{\nu}$ be the associated quasi-regular representation of $G$ on $L^2(X,\nu)$. Let $L$ be a length function on $G$ and let $(X,d)$ be a metric space inducing the topology of $X$. For every $n \in \mathbb{N}$, let $E_n$ be a symmetric finite subset of $G$, that is $E_n = E_n^{-1}$, and $e_n : E_n \longrightarrow X$ be a map. Assume that the following conditions hold:
	\begin{itemize}
		\item[(1)] for every $g \in G$, $\|\pi_{\nu}(g)\mathds{1}_X\|_{L^{\infty}(X,\nu)} < \infty$,
		\item[(2)] $\lim_{n \rightarrow \infty}|E_n| = \infty$,
		\item[(3)] for all Borel subsets $W,V \subset X$ such that $\nu(\partial W) = \nu(\partial V) = 0$, $$\limsup_{n \rightarrow \infty}\frac{1}{|E_n|}|\left\{\gamma \in E_n: e_n(\gamma^{-1})\in W ~\mbox{and}~e_n(\gamma)\in V \right\}| \leq \nu(W)\nu(V),$$
		\item[(4)] for every $r \ge 0$, there is a non-increasing function $h_{r}: [0,\infty) \longrightarrow [0,\infty) $ such that $\lim_{s \rightarrow \infty}h_{r}(s) = 0$ and such that $$\forall n \in \mathbb{N}, \forall \gamma \in E_n, \frac{\langle \pi_{\nu}(\gamma)\mathds{1}_{X}, \mathds{1}_{\{x\in X: d(x,e_n(\gamma))\geq r\}}\rangle_{L^2}}{\Phi(\gamma)} \leq h_{r}(L(\gamma)),$$
		\item[(5)] $$\sup_{n}\left\|M^{\mathds{1}_{X}}_{E_n}\mathds{1}_{X}\right\|_{L^{\infty}(X,\nu)} < \infty.$$
	\end{itemize}
	Then the quasi-regular representation $\pi_{\nu}$ is ergodic with respect to $(E_n, e_n)$ and any $f \in \bar{H}^{L^{\infty}(X,\nu)}$, where $H$ is a vector space generated by $$ \{\mathds{1}_{U}:\nu(\partial U) = 0 ~\mbox{and}~ U ~\mbox{is a Borel subset of}~ X\}.$$
\end{theorem}
\begin{remark}
	Thanks to the condition (1), the Harish-Chandra function is defined for each $\gamma$ in $G$.
\end{remark}
\begin{proposition}[\cite{BaderMuchnik}]\label{Irre}
	Under the assumptions in the above theorem, if moreover $\nu$ is a Radon measure, then $\pi_{\nu}$ is irreducible.
\end{proposition}
\subsection{Boundary representations of mapping class groups.}
We now consider quasi-regular representations of mapping class groups and state our main theorem. For more on mapping class groups and Teichm\"{u}ller theory, we refer to \cite{Farb2012}, \cite{ABEM} and \cite{Kerckhoff80}. \\

\noindent{\bf Mapping class groups and Teichm\"{u}ller spaces.}~~~ Let $S =S_{g}$ be a genus $g$, closed, connected, orientable surface. We always assume that $g \geq 2$. All arguments here work for hyperbolic surfaces with punctures as well. The {\it mapping class group} $\M(S)$ of $S$ is the group of isotopy classes of orientation-preserving homeomorphisms of $S$. Namely, if the group of orientation-preserving homeomorphisms of $S$ is denoted by ${\rm Homeo}^{+}(S)$ and the group of homeomorphisms of $S$ that isotopic to the identity is denoted by ${\rm Homeo}_{0}(S)$, then $$ \M(S) = {\rm Homeo}^{+}(S)/ {\rm Homeo}_{0}(S).$$

 We remark here that mapping class groups of surfaces are finitely presented and considered to be discrete groups. The {\it Teichm\"{u}ller space} $\T(S)$ of $S$ is the space of homotopy classes of hyperbolic structures. The Teichm\"{u}ller space $\T(S)$ is homeomorphic to $\mathbb{R}^{6g-6}$ and the mapping class group $\M(S)$ acts on $\T(S)$ by changing markings. The quotient $\mathcal{M}(S) = \T(S)/\M(S) $ is the {\it moduli space} of $S$. There are several distances on $\T(S)$ so that $\M(S)$ acts as isometries and the one that we will use is the {\it Teichm\"{u}ller distance} $d = d_{T}$. It is defined as follows: For $\mathcal{X}=[(X,\phi)],\mathcal{Y}=[(Y,\psi)] \in \T(S)$, $d(\mathcal{X},\mathcal{Y}) = \frac{1}{2}\log K_f$, where $f: X \longrightarrow Y$ is the Teichm\"{u}ller mapping, locally in the form of $x+iy \mapsto e^{t}x + ie^{-t}y$, in the isotopy class of $\psi\circ\phi^{-1}$, namely the quasi-conformal homeomorphism with minimal dilatation in the isotopy class of $\psi\circ\phi^{-1}$ and $K_f$ is the dilatation of $f$. It is obvious that $\M(S) \subset Isom(\T(S),d)$. Neither the Teichm\"{u}ller space $\T(S)$ nor $\mathcal{M}(S)$ is compact.\\
\noindent{\bf Measured foliations.}~~~ The Teichm\"{u}ller space can be compactified in several ways. The compactification we will use in this paper is the Teichm\"{u}ller compactification.  Fix a point $o \in \T(S)$ that is considered to be a Riemann surface $X$ via uniformization. A {\it holomorphic quadratic differential} $q \in {\rm H^{0}(X, \Omega_{X}^{\otimes 2})}$ on $X$ is locally of the form $q(z)dz^2$ such that $q(z)$ is a holomorphic function. Define a {\it norm} on $q$ by $$\|q\| = \int_{X}|q(z)|dxdy$$ and consider the unit open ball $B^{1}(X)$ with respect to $\| \cdot \|$. The set $QD(X)$ of holomorphic quadratic differentials is a vector space and can be identified with the cotangent space of $\T(S)$ at $o$. There is a homeomorphism $\pi : B^{1}(X) \longrightarrow \T(S)$ sending each open unit ray in $QD(X)$ starting at the origin to a Teichm\"{u}ller geodesic starting at $o$. The {\it Teichm\"{u}ller compactification} is then the visual compactification by adding ending points in the unit sphere of $QD(X)$ to each ray. The Teichm\"{u}ller compactification will be denoted by $\bar{\T(S)}$. Thus, the boundary $\partial \bar{\T(S)}$ of $\bar{\T(S)}$ is the unit sphere $QD^{1}(X)$.\\

One could give a geometric description of $\partial \bar{\T(S)}$ via projective measured foliations. A {\it measured foliation} on $S$ is a singular foliation of $S$ endowed with a transverse measure. The space $\MF(S)$ of measured foliations is then the set of equivalent classes of measured foliations where the equivalence is given by Whitehead moves and isotopy. The space $\MF(S)$, endowed with the weak topology on measures, is homeomorphic to $\mathbb{R}^{6g-6}$. The quotient, called {\it the projective measured foliation} space $\PMF(S)$ of $S$, of $\MF(S)$ by the nature action of $\mathbb{R}_{+}$ is homeomorphic to the $6g-7$ sphere $S^{6g-7}$. Both $\MF(S)$ and $\PMF(S)$ are equipped with a $\M(S)-$action. There is a deep relation between $\MF(S)$ and $QD(X)$. Namely, for each holomorphic quadratic differential $q$, the {\it vertical measured foliation} $\mathcal{V}(q)$ of $q = q(z)dz^2$ is the foliation given by the integral curves of the holomorphic tangent vector field on $S$ such that each vector has a value in negative real numbers under $q$, where the transverse measure is given by integration of $|Re\sqrt{q}|$. By a theorem of Hubbard-Masur, the map $\mathcal{V}$ that assigns each holomorphic quadratic differential $q$ on $X$ to $\mathcal{V}(q)$ is a homeomorphism from $QD(X)-\{0\}$ onto $\MF(S)$. The composition $\pi \circ \mathcal{V}$ of the map $\mathcal{V}: QD(X)-\{0\} \longrightarrow \MF(S)$ and the quotient map $\pi: \MF(S) \longrightarrow \PMF(S)$ gives the identification of $QD^{1}(X)$ with $\PMF(S)$. Thus, we will regard $\PMF(S)$ as the boundary of the Teichm\"{u}ller compactification of $\T(S)$. The equivalent class of $\xi \in \MF(S)$ in $\PMF(S)$ will be denoted by $[\xi]$. Any $q \in QD^1(X)$ (hence $[\mathcal{V}(q)] \in \PMF(S)$) determines a Teichm\"{u}ller geodesic ray $g_t$ starting from $o$, hence, by abuse of terminologies, we will call $q$ and $[\mathcal{V}(q)]$ the {\it direction} of $g_t$ and sometimes write $g_t$ as $g^q_t$ or $\mathcal{V}(q)(t)$.  \\

Any isotopy class $\gamma$  of essential simple closed curves on $S$ defines a (topological) foliation $\lambda(\gamma)$. Hence, any weighted isotopy class $c\gamma$ of essential simple closed curves defines a foliation $\lambda \in \MF(S)$. The measured foliation $\lambda$, as topological foliation, is the same as $\lambda(\gamma)$, but the transverse measure is given by $c$. Therefore, let $\mathcal{C}(S)$ denote the set of isotopy classes of essential simple closed curves, there is an embedding of $\mathcal{C}(S) \times \mathbb{R}_{+}$ into $\MF(S)$. The image is dense (\cite{Thurston1988}). This embedding enable us to define three functions that we will use. The first one is the intersection number on $\MF(S)$. The intersection number $i : \MF(S) \times \MF(S)\longrightarrow \mathbb{R}_{+}$ is the unique continuous function on $\MF(S) \times \MF(S)$ that extends the geometric intersection number of two essential simple closed curves and satisfies $i(c\lambda,\xi)=ci(\lambda,\xi)$ for every $c>0$ (see \cite[Corollary 1.11]{Rees}). The second one is the extremal length. Let $o = [(X,\phi)] \in \T(S)$ where $X$ is a Riemann surface. Let $\gamma$ be the isotopy class of an essential simple closed curve. The extremal length $\E_{X}(\gamma)$ of $\gamma$ in $X$ is defined to be $$\E_{X}(\gamma) = \sup_{\rho}\ell_{\rho}(\gamma)^2,$$ where $\rho$ runs over all metrics with unit area in the conformal class of $X$ and $\ell_{\rho}(\gamma)$ is the infimum of $\rho-$length of simple closed curves in $\gamma$. Then the extremal length $\E_{X}: \MF(S) \longrightarrow \mathbb{R}_{+}$ is the unique continuous function on $\MF(S)$ that extends the extremal length of $\mathcal{C}(S)$ and satisfies $\E_{X}(c\lambda) =c^2\E_{X}(\lambda)$ for $c \in \mathbb{R}_{+}$ (see \cite[Proposition 3]{Kerckhoff80}). Note that the extremal length in fact is defined on $\T(S) \times \MF(S)$, namely, if $[(X,\phi)] = [(Y,\psi)] \in \T(S)$, then $\E_{X}(\cdot) = \E_{Y}(\cdot)$. So we will write $\E_{o}(\cdot)$ rather than $\E_{X}(\cdot)$ for $o = [(X,\phi)]$. The third one is the hyperbolic length $\ell_{o}(\gamma)$ which is defined to be the $X-$length of unique $X-$hyperbolic geodesic $\tilde{\gamma}$ in the isotopy class $\gamma$. The function $\ell_{o}(\cdot)$ can be uniquely extended as well to $\MF(S)$ to obtain a continuous function $\ell_{o}$ on $\MF(S)$ \cite{Kerckhoff85}. We will use the following relation (see \cite{ABEM}): given a point $o$ in $\T(S)$, then there exists a constant $C = C(o)$, depending on $o$, such that $$ \forall \xi \in \MF(S), \frac{1}{C} \ell_o(\xi) \leq \sqrt{\E_{o}(\xi)} \leq C\ell_o(\xi).$$ \\

Recall that a measured foliation $\lambda$ is called {\it minimal} if it has no simple closed leaves. Two measured foliations are said to be {\it topologically equivalent} if they, as topological foliations, are differ by isotopies and Whitehead moves. A measured foliation $\xi$ is called {\it uniquely ergodic} if it is minimal and any measured foliation $\zeta$ that topologically equivalent to $\xi$ is measure equivalent to $\xi$, that is, $[\xi]=[\zeta]$. When $\xi$ is uniquely ergodic, we will call $[\xi]$ uniquely ergodic. It is well-known that the set of uniquely ergodic measured foliations has full measure with respect to the Thurston measure defined later. The following two lemmas are essential to our approach using intersection numbers.
\begin{lemma}[Theorem 1.12 in \cite{Rees}, \cite{Masur82}]\label{lemma:uniquelyergodic}
	Let $\lambda$ be a uniquely ergodic measured foliation and $\eta$ be any measured foliation. Then $i(\lambda,\eta) = 0$ if and only if $[\lambda] = [\eta]$.
\end{lemma}
\begin{lemma}[Masur's criterion \cite{Masur92}]\label{lemma:masurcriterion}
	Given $\epsilon > 0$. If a Teichm\"{u}ller geodesic ray $g_t$ starting from $o$ does not leave $\T_{\epsilon}(S)$ eventually, then the direction of $g_t$ is uniquely ergodic.
\end{lemma}
One feature of the Teichm\"{u}ller compactification is that the action of $\M(S)$ cannot be extended continuously to $\bar{\T(S)}$ \cite{Kerckhoff80}. However, uniquely ergodic measured foliations are nice points in terms of $\M(S)-$action in the following sense.
\begin{lemma}[\cite{Masur82}]\label{smoothness}
	The mapping class group acts continuously on $\bar{\T(S)}$ at uniquely ergodic points on the boundary.
\end{lemma}
The following formula proved by Kerckhoff concerning the calculation of Teichm\"{u}ller distances will be used frequently.
\begin{lemma}[\cite{Kerckhoff80}, Kerckhoff's formula]\label{lemma:Kerckhoffformula}
	$$\forall x,y \in \T(S), d_T(x,y) = \frac{1}{2}\sup_{[\xi] \in \PMF(S)} \ln \left(\frac{\E_{x}(\xi)}{\E_{y}(\xi)}\right).$$
\end{lemma}

\noindent{\bf Hyperbolicity.}~~~ It was first proved in \cite{MasurWolf} that the Teichm\"{u}ller space $(\T(S),d_T)$ is not hyperbolic in the sense of Gromov. However some triangles in $(\T(S),d_T)$ are indeed thin. We now collect several related results in order to compare neighborhoods in $\PMF(S)$ defined by projections of balls in $\T(S)$ and the ones defined by intersection numbers.\\

Recall that, for $\epsilon > 0$, the {\it $\epsilon-$thick part} $\T_{\epsilon}(S)$ of the Teichm\"{u}ller space $\T(S)$ is defined to be $$\T_{\epsilon}(S) = \{y\in \T(S): \forall c \in \mathcal{C}(S),\E_{y}(c) \geq \epsilon\}.$$
The first result, generalizing a theorem of Rafi \cite{Rafi14}, describes when triangles are thin. We denote $ \mathcal{N}_D(A)$ for a subset $A$ of $\T(S)$ by the $D-$neighborhood of $A$. Recall that {\it a geodesic segment $I:[a,b] \rightarrow \T(S)$ has at least proportion $\theta$ in $\T_{\epsilon}(S)$} if $$Thk^{\%}_{\epsilon}[I]\doteq \frac{|\{a \leq s \leq b: I(s) \in \T_{\epsilon}(S)\}|}{b-a} \geq \theta.$$ 
\begin{theorem}[\cite{DDM}]\label{theorem:DDM}
	Given $\epsilon > 0$ and $0 < \theta \leq 1 $, there exist constants $D = D(\epsilon,\theta), L_0=L_0(\epsilon,\theta)$ such that if $I \subset [x,y]$ is a geodesic subinterval in $\T(S)$ of length at least $L_0$	and at least proportion $\theta$ of $I$ is in $\T_{\epsilon}(S)$, then for every $z \in \T(S)$, we have $$ I \cap \mathcal{N}_{D}([x,z]  \cup [y,z]) \ne \emptyset.$$
\end{theorem}
The following result will also be used later. Recall that two parametrized geodesics segment $\delta(t)$ and $\delta'(t)$ defined on $[a,b]$ are said to {\it $P-$fellow travel} in a parametrized fashion if, for every $t\in [a,b]$, $d_T(\delta(t),\delta'(t)) \leq P$.
\begin{theorem}[\cite{Rafi14}]\label{theorem:fellowtravelling}
	Let $\epsilon > 0$ and $R > 0$. Then there exists $P =P(\epsilon, R) > 0$ such that whenever $x_1,x_2,y_1,y_2$ are in $\T_{\epsilon}(S)$ with $$d_T(x_1,x_2) \leq R,d_T(y_1,y_2) \leq R,$$ the geodesic segment $[x_1,y_1]$ and $[x_2,y_2]$ are $P-$fellow travelling. 
\end{theorem}

\noindent{\bf Boundary representations of mapping class groups.}~~~ We are in a position to discuss a special type of quasi-regular unitary representations of mapping class groups. Fix $o \in \T(S)$, we first define a Radon measure $\nu_{o}$ on $\PMF(S)$. Let $\nu_{Th}$ be the Thurston measure on $\MF(S)$. For any open subset $U \subset \PMF(S),$ one defines $\nu_{o}(U)$ to be $$\nu_{o}(U) = \nu_{Th}\left(\{\xi: [\xi] \in U, \E_{o}(\xi) \leq 1\}\right).$$ One could verify that $\forall \gamma \in \M(S), \gamma_{*}\nu_{o}=\nu_{\gamma.o}$ and $[\nu_{x}] = [\nu_{y}],\forall x,y \in \T(S)$. Therefore, one has $$\forall x,y \in \T(S), [\xi]\in \PMF(S), \frac{d\nu_{x}}{d\nu_{y}}([\xi])= \left(\frac{\E_{y}(\xi)}{\E_{x}(\xi)}\right)^{\frac{6g-6}{2}}.$$ By the definition of extremal length, the function $[\xi] \mapsto \left(\frac{\E_{y}(\xi)}{\E_{x}(\xi)}\right)^{\frac{6g-6}{2}}$ is well-defined on $\PMF(S)$. We have, in particular, $$\forall \gamma \in \M(S), [\xi]\in \PMF(S), \frac{d\gamma_{*}\nu_{o}}{d\nu_{o}}([\xi])= \left(\frac{\E_{o}(\xi)}{\E_{\gamma.o}(\xi)}\right)^{\frac{6g-6}{2}}.$$ Hence one has a quasi-regular unitary representation $\pi_{\nu_{o}}$ of $\M(S)$ on the Hilbert space $L^{2}(\PMF(S),\nu_{o})$. The quasi-regular representation $\pi_{\nu_{o}}$ of $\M(S)$ is called the {\it boundary representation} of $\M(S)$ (with respect to $o$).\\

As intersection numbers will be the main tool, we embed $\PMF(S)$ into $\MF(S)$. For each $[\xi]$, define $\tau(\xi) \in \MF(S)$ to be the unique element in $[\xi]$ such that $\E_{o}\left(\tau(\xi)\right) = 1.$ Hence, the map $\tau : \PMF(S) \longrightarrow \MF(S)$ is a section of the projection $\pi : \MF(S) \longrightarrow \PMF(S).$  When talking about intersection numbers for two points in $\PMF(S)$, except in Section \ref{sec:regularitypants}, we will use the image of $\tau$.\\
\noindent {\bf Ergodic boundary representation.}~~~ From now on, let $S =S_{g} (g \geq 2)$ be a genus $g$ closed, orientable surface and fix a point $o=[(X,\phi)]  \in \T(S)$. Normalize $\nu_o$ to be a probability measure. Denote $h = 6g - 6$ and let $\epsilon > 0$ and $\theta > 0$. Let $L$ be the length function on $G$ induced by the Teichm\"{u}ller distance $d_T$, namely $L(g) = d_T(o, g\cdot o)$.\\

Inspired by \cite{gekhtman2013stable} and \cite{DDM}, we first describe our choice of $E_n$ that fits in Theorem \ref{criterion2}. Let $g^q_t$ be a Teichm\"{u}ller geodesic ray starting from $o$ in the direction of $q \in QD^1(X)$. For every $m > 0$, recall that $$Thk^{\%}_{\epsilon}[o,g_m]\doteq \frac{|\{0 \leq s \leq m: g_s \in \T_{\epsilon}(S)\}|}{d_T(o,g_m)}.$$ 
\begin{theorem}[\cite{DDM} Proposition 5.5] \label{theorem:DowDucMas}
	For all $0 < \theta < 1$, there exists $\epsilon > 0$ such that for all $o = (X,\phi) \in \T(S)$ $$\lim_{R_0 \rightarrow \infty}\nu_o\left( \{q \in QD^1(X):  Thk^{\%}_{\epsilon}[o,g^q_m] \geq \theta , \forall m > R_0 \}\right) = 1.$$	
\end{theorem}

We then fix any $\theta \in (0,1)$ and take $\epsilon > 0$ sufficiently small. We identify $QD^1(X)$ with $\PMF(S)$ and $g^q_t$ with $\mathcal{V}(q)(t)$. For each $R > 0$, we define $$ U(R,\theta,\epsilon) = \{\xi \in \PMF(S): Thk^{\%}_{\epsilon}[o,\xi_m] \geq \theta, \forall m >R\}.$$ Then if $R_2 \geq R_1 > 0$, we have $$ U(R_1,\theta,\epsilon) \subset U(R_2,\theta,\epsilon).$$ Define $$U(\theta, \epsilon) = \cup_{R > 0} U(R,\theta,\epsilon),$$ then, by Theorem \ref{theorem:DowDucMas}, one has $$\nu_o(U(\theta,\epsilon))) = 1.$$ Furthermore, after a suitable choice of $\theta$, one has $\nu_o(\partial (U(\theta,\epsilon))) = 0$ and by Masur's criterion (Lemma \ref{lemma:masurcriterion}), the set $U(\epsilon,\theta)$ consists of uniquely ergodic directions. We now fix the choice of $\epsilon$ and $\theta$ and for $\gamma \in \M(S)$, denote the direction determined by the oriented geodesic $[o,\gamma\cdot o]$ by $\xi_{\gamma}$. Now we are in a position to describe $E_n$. Fix $\rho > 0$ and let $L_0 = L_0(\theta,\epsilon)$ be the constant as Theorem \ref{theorem:DDM}. For $\frac{1}{3h}\ln \ln n > \max{\{L_0,\rho\}}$,  define \textbf{THE SET $\mathcal{E}(\theta,\epsilon, n, o, \rho)$} to be the set of all elements $\gamma$ in $\M(S)$ satisfying:
\begin{itemize}
	\item[(a)] $d(\gamma \cdot o,o) \in (n- \rho, n+\rho)$;
	\item[(b)] Both $\xi_{\gamma}$ and $\xi_{\gamma^{-1}}$ are in $U(\theta, \epsilon)$;
	\item[(c)] If $g(t)$ is either the geodesic ray $\xi_{\gamma}(t)$ or $\xi_{\gamma^{-1}}(t)$, then the segment $[o, g(\frac{1}{3h}\ln \ln n)]$ has at least proportion $\theta$ in $\T_{\epsilon}(S)$.
\end{itemize}

\begin{lemma}\label{lemma:longsegment}
	Let $n$ large enough as before. Then for $\gamma \in \mathcal{E}(\theta,\epsilon, n, o, \rho)$, there exists a geodesic segment $I_{\gamma}$ of length $\frac{1}{3h}\ln \ln n$ in the geodesic $[o,\gamma \cdot o]$ that has at least proportion $\theta$ in $\T_{\epsilon}(S)$ and containing $\gamma \cdot o$. 
\end{lemma}
\begin{proof}
	Let $\gamma \in \mathcal{E}(\theta,\epsilon, n, o, \rho)$. Since the geodesic ray $\xi_{\gamma^{-1}}(t)$ satisfies (c) in the definition of $\mathcal{E}(\theta,\epsilon, n, o, \rho)$, the first segment $I_{\gamma^{-1}}$ of $[o,\gamma^{-1} \cdot o]$ of length $\frac{1}{3h}\ln \ln n$ has at least proportion $\theta$ in $\T_{\epsilon}(S)$. As $\gamma \cdot [o, \gamma^{-1} \cdot o] = [\gamma \cdot o,o]$, therefore, the geodesic $[o,\gamma \cdot o]$ has a subinterval $\gamma \cdot I_{\gamma^{-1}}$ of length $\frac{1}{3h}\ln \ln n$ that has at least proportion $\theta$ in $\T_{\epsilon}(S)$ and starting at point $\gamma \cdot o$.   
\end{proof}
In the next section, we will prove that $\mathcal{E}(\theta,\epsilon, n, o, \rho)$ has exponential growth. We first state one obvious property of the boundary representation.
\begin{lemma}\label{essentialbound}
	Let $\pi_{\nu}$ be the boundary representation of $\M(S)$. For every $g \in \M(S)$, $\|\pi_{\nu_o}(g)\mathds{1}_{\PMF(S)}\|_{L^{\infty}(\PMF(S),\nu_o)} < \infty$
\end{lemma}
\begin{proof}
	The lemma is an easy consequence of Kerckhoff's formula, namely Lemma \ref{lemma:Kerckhoffformula}, on Teichm\"{u}ller distances. By Lemma \ref{lemma:Kerckhoffformula},  $$\forall x,y \in \T(S), \forall [\xi] \in \PMF(S), \left(\frac{\E_{x}(\xi)}{\E_{y}(\xi)}\right)^{\frac{1}{2}} \leq e^{d_{T}(x,y)}.$$ As $\pi_{\nu_o}(g)\mathds{1}_{\PMF(S)} = \left(\frac{\E_{o}(\xi)}{\E_{g\cdot o}(\xi)}\right)^{\frac{6g-6}{4}},$ one has $$\|\pi_{\nu_{o}}(g)\mathds{1}_{\PMF(S)}\|_{L^{\infty}(\PMF(S),\nu_{o})} \leq e^{\frac{6g-6}{2}d_T(o,g \cdot o)} < \infty.$$
\end{proof}
The following theorem is a slight variant of Theorem \ref{criterion} whose proof is close to its original proof.
\begin{theorem}\label{criterion2}
	Let $\pi_{\nu_o}$ be the associated quasi-regular representation of $\M(S)$ on $L^2(\PMF(S),\nu_o)$. Let $i(\cdot,\cdot)$ be the intersection number function on $\PMF(S)$. Let $n \gg \rho$ and let $E_n = E_n(\rho) \subset \{g \in \M(S): d_T(o,g\cdot o) \in [n-\rho,n+\rho]\}$ be symmetric. Let $e_n = Pr: E_n \longrightarrow \PMF(S)$ be the radial projection from $o$. Assume that the following conditions hold:
	\begin{itemize}
		
		\item[(1)] $\lim_{n \rightarrow \infty}|E_n| = \infty$,
		\item[(2)] for all Borel subsets $W,V \subset \PMF(S)$ such that $\nu_o(\partial W) = \nu_o(\partial V) = 0$, $$\limsup_{n \rightarrow \infty}\frac{1}{|E_n|}|\left\{\gamma \in E_n: e_n(\gamma^{-1})\in W ~\mbox{and}~e_n(\gamma)\in V \right\}| \leq \nu_o(W)\nu_o(V),$$
		\item[(3)] for every $n \gg \rho$, there are two sequences of reals $\{h_{r_n}(n, \rho)\}$ and $\{r_n\}$ such that $\lim_{n \rightarrow \infty}h_{r_n}(n,\rho) = \lim_{n\rightarrow \infty}r_n = 0$ and such that $$\forall n \in \mathbb{N}, \forall \gamma \in E_n, \frac{\langle \pi_{\nu_o}(\gamma)\mathds{1}_{\PMF(S)}, \mathds{1}_{\{x\in \PMF(S):~i(x,e_n(\gamma))\geq r_n\}}\rangle}{\Phi(\gamma)} \leq h_{r_n}(n,\rho),$$
		\item[(4)] $$\sup_{n}\left\|M^{\mathds{1}_{\PMF(S)}}_{E_n}\mathds{1}_{\PMF(S)}\right\|_{L^{\infty}(\PMF(S),\nu_o)} < \infty.$$
	\end{itemize}
	Then the quasi-regular representation $\pi_{\nu_o}$ is ergodic with respect to $(E_n, e_n)$ and any $f \in \bar{H}^{L^{\infty}(\PMF(S),\nu_o)}$, where $$H = <\mathds{1}_{U}:\nu_o(\partial U) = 0 ~\mbox{and}~ U ~\mbox{is a Borel subset of}~ \PMF(S)>.$$
\end{theorem}
\begin{remark}
	As Theorem \ref{criterion2} is slightly different from Theorem \ref{criterion}, we should make a few comments. One could easily find the only difference is the point (3) here since the original point (1) has been replaced automatically by Lemma \ref{essentialbound}. The assumption (3) in Theorem \ref{criterion2}  is slightly weaker than the assumption (4) in Theorem \ref{criterion} as we don't require $h_{r}$ to be independent on $L(g)$. The cost of this weaker assumption is that, in order to make it easy to state, we should assume that elements in $E_n$ has almost the same induced distance $n$ with $id \in \M(S)$. 
\end{remark}
\begin{proof}
As the proof of in \cite[Theorem 2.2]{BLP} works quite well in our situation except small modifications, we only point out the necessary modifications here. The main issue here is that the intersection number function $i(\cdot,\cdot)$ does not define a metric structure, even not the usual topological structure on $\PMF(S)$ since there are points $\xi \neq \eta$ with $i(\eta,\xi) = 0$. However, there is a full $\nu_o-$measure subset $\mathcal{UF} \subset \PMF(S)$ consisting of uniquely ergodic measured foliations such that, for every metric on $\PMF(S)$, the induced topology coincides with the topology induced by the intersection number functions when restrict to $\mathcal{UF}$ by Lemma \ref{lemma:uniquelyergodic}.  

The proof of \cite[Page 2037, Theorem 2.2]{BLP} uses \cite[Lemma 2.20]{BLP} and \cite[Proposition 2.21]{BLP}, we point out modifications in these two places and the rest is the same as the proof in \cite[Theorem 2.2]{BLP}.

As the proof eventually based on \cite[Lemma 2.19]{BLP}, our proof also eventually based on it and the followings are our assumptions. The set $B = \PMF(S)$ is equipped with any metric $d$ which is compatible the quotient topology and whenever talking about Borel subsets $U,W$, we mean a compact subsets with null $\nu_o-$zero boundaries. These assumptions are reasonable since all statements in \cite[Lemma 2.19]{BLP} is stable under taking closures. Note that although working on metric measure spaces, the metric is only used in the following way: for subsets $U, W$, $d(U,W) > 0$ iff $\overline{U} \cap \overline{W} = \emptyset$.

Define open sets $W(r)$ for $r > 0$, 
\begin{equation*}
    \begin{aligned}
        &W(r) = \{\eta \in \PMF(S): \exists ~\xi \in \mathcal{UF} \cap W,~i(\eta, \xi) < r\} \\
        &=\cup_{\xi \in \mathcal{UF} \cap W}\{\eta: i(\eta,\xi) < r\}.
        \end{aligned}
\end{equation*}
For\cite[Lemma 2.20]{BLP}, consider $W(r_n)$ rather than $W(r)$ where $r_n$ is the sequence in (3), then take the limit $\lim_{n \to \infty}h_{\tau_n}(n,\rho)$ rather than $\lim_{s \to \infty}h_r(s)$.

For \cite[Proposition 2.21]{BLP}, consider $W(r_n)$ and $V(r_n)$ rather than $W(r)$ and $V(r)$ in the proof in \cite[Page 2037, Proposition 2.21]{BLP}. Firstly one has, for $r \geq 0$,
$$\overline{W(r)} = \cup_{\xi \in W}\{\eta: i(\eta,\xi) \leq r\}.$$
Then using the same notations,
\begin{equation*}
    \begin{aligned}
     & \limsup_{n\to \infty}\langle M^{\mathds{1}_U}_{(E_n,e_n)}\mathds{1}_{B}, \mathds{1}_{W}\rangle =\limsup_{n\to \infty}\frac{1}{|E_n|}\int_{E_n}\mathds{1}_{U}(e_n(g))\mathds{1}_{W(r_n)}(e_n(g))dg \\
     &\leq\limsup_{n\to \infty}\frac{1}{|E_n|}\int_{E_n}\mathds{1}_{U}(e_n(g))\mathds{1}_{\overline{W(r_n)}}(e_n(g))dg.
    \end{aligned}
\end{equation*}
Since $$\forall r \leq r',~~~\overline{W(r)} \subset \overline{W(r')},$$
so if $\nu_o(\partial \overline{W(r_k)}) = 0$ for all $k$, then by the second condition (2),
\begin{equation*}
    \begin{aligned}
        \limsup_{n\to \infty}\langle M^{\mathds{1}_U}_{(E_n,e_n)}\mathds{1}_{B}, \mathds{1}_{W}\rangle \leq \nu_o(U \cap \overline{W(r_k)}), \forall k.
               \end{aligned}
\end{equation*}
Hence if $$\lim_{k \to \infty}\nu_o(\overline{W(r_k)} \cap U) \leq \nu_o(\overline{W(0)} \cap U),$$ then one can continue the same argument as \cite[Proposition 2.21]{BLP} since $U  \cap \overline{W(0)}$ has only non uniquely ergodic points (by the assumption that $U \cap W = \emptyset$) and $\nu_o$ is supported on uniquely ergodic points. So we are leave to prove the following lemma which is an analogue of \cite[Proposition 2.8]{BLP} and then change $r_k$ in the assumption a little bit.  
\end{proof}
\begin{lemma}
Use the notations in the proof above. We have
\begin{itemize}
    \item[(a)] The set $\{r > 0, \nu_o(\partial \overline{W(r)}) = 0\}$ is dense.
    \item[(b)] $$\lim_{k \to \infty}\nu_o(\overline{W(r_k)} \cap U) \leq \nu_o(\overline{W(0)} \cap U).$$
    \end{itemize}
\end{lemma}
\begin{proof}
Assume $W$ is compact, otherwise replace it by its closure. The part (b) is simply an application of Monotone Convergence Theorem since $\overline{W(r_{k+1})} \cap U  \subset \overline{W(r_{k})} \cap U$ for a decreasing sequence $\{r_k\}$. For the part (a), we argue in the same way as \cite[Proposition 2.18]{BLP}. Also assume $W$ is compact, thus for $\xi \in \PMF(S)$, define $s(\xi,W) = \min{\{i(\xi,w),w \in W \}}$. Therefore $$\forall r > 0, \partial \overline{W(r)} \subset \{\eta: s(\eta, W) = r\}.$$ So for $r \neq r', \partial \overline{W(r)} \cap \partial \overline{W(r')} = \emptyset$. Hence if the part (a) is wrong, then it will contradict with the fact that $\nu_o$ has finite measure.

\end{proof}
 Our main result is the following theorem.
\begin{theorem}\label{theoreom:ergodicitymcg}
	There exist $\theta$ and $\epsilon$ such that, if $E_n = \mathcal{E}(\theta,\epsilon, n, o, \rho)$, which is described before Lemma \ref{lemma:longsegment} (and up to passing to asubsequence), then the boundary representation $\pi_{\nu_o}$ is ergodic with respect to $(E_n, Pr)$ and any $f \in \bar{H}^{L^{\infty}}$ as above. In other words, the pair $(E_n = \mathcal{E}(\theta,\epsilon, n, o, \rho), Pr)$ satisfies all conditions listed in Theorem \ref{criterion2}.
\end{theorem}
As $\nu_o$ is a Radon measure, one has immediately the following two corollaries by Proposition \ref{Irre} and Remark \ref{remark:quasiregularirreducible}.
\begin{corollary}\label{corollary:mappingclassgroupirreducibility}
	The boundary representation $\pi_{\nu_o}$ of $\M(S)$ is irreducible.
\end{corollary}
\begin{corollary}
	The mapping class group $\M(S)$ acts ergodically on $\PMF(S)$ with respect to the measure class $[\nu_o]$.	
\end{corollary}
We then mention a property of the boundary representation $\pi_{\nu_o}$. Recall that a unitary representation of a group $G$ is called {\it tempered} if it is weakly contained in the regular representation $L^{2}(G)$.
\begin{proposition}\label{proposition:tempered}
	The boundary representation $\pi_{\nu_o}$ of $\M(S)$ is tempered.
\end{proposition}
\begin{proof}
	We argue as \cite[Proposition 6.3]{garncarek2016boundary}. By the main theorem in \cite{Gabriella}, we need to verify that the action of $\M(S)$ on $\PMF(S)$ is amenable. This is in \cite[Proposition 8.1]{Hamenstaedt2009a} as a corollary of topological amenability of the action of $\M(S)$ on $\PMF(S)$.
\end{proof}

\noindent {\bf Notations.}~~~ We make some conventions that we will use in the sequel. 
\begin{itemize} 
	\item $S = S_g$: a genus $g \geq 2$, closed, oriented, connected surface.
	\item $h = 6g-6$.
	\item $o$: the base point in $\T(S)$ which is chosen to be generic in the sense that $Stab_{o}(\M(S)) = {id}$. Denote $\nu =\nu_o$ and the measure is normalized so that $\nu (\PMF(S)) = 1$.
	\item The projective measured foliation space $\PMF(S)$ is regarded as a subset of $\MF(S)$ by $\tau$ and an element $[\xi]$ in $\PMF(S)$ is then written as $\xi$, so both $[\xi]$ and $\xi$ will be called directions when there are no confusions.
	\item Fix arbitrary $\rho > 0$ and assume $n \gg \rho$.
	\item  $Pr_{y} : \T(S)-\{y\} \longrightarrow \PMF(S)$: the radial projection from $\T(S)$ to $\PMF(S)$ that assigns every point $z \in \T(S) -\{y\}$ to the vertical measured foliation of the unit quadratic differential defined by the oriented geodesic $[y,z]$. For $y = o$, we simply denote $Pr_{o}$ to be $Pr$.
	\item $B(y,R)$: the closed ball in $\T(S)$ of radius $R$ at $y$  with respect to the Teichm\"{u}ller distance $d = d_T$.
	\item $\asymp$: if $A(t),B(t)$ are two functions, we use the notation $A \asymp B$ to mean $\frac{A(t)}{B(t)}\rightarrow 1$ as $t \rightarrow \infty$ and $A \widetilde{<} B$ to mean $\lim_{t \rightarrow \infty}\frac{A(t)}{B(t)} \leq 1$. The notation $A \widetilde{>} B$ is defined similarly.  
	\item $A \sim_{\theta} B$: there is multiplicative constants $C_1 > 0,C_2 > 0$ depending on $\theta$ so that $$C_1A \leq B \leq C_2A.$$  $A \prec_{\theta} B$: there is a multiplicative constant $D=D(\theta) > 0$ so that $$A \leq DB.$$ And  $A \succ_{\theta} B$ is defined similarily.
	\item Denote $U = U(\theta,\epsilon)$ and $E_n = \mathcal{E}(\theta,\epsilon, n, o, \rho)$ in the sequel which is described before Lemma \ref{lemma:longsegment} with $\theta > 0.999$.
	\item $\xi_{\gamma} \in \PMF(S)$ (for $\gamma \in \M(S)-\{id\}$): the direction of the oriented geodesic segment $[o,\gamma \cdot o]$. 
\end{itemize}

\section{Exponential growth and Shadow lemma}

\subsection{Exponential growth.}
In this subsection, we will show that $|E_n|$ goes to infinity. In fact, we will show that $|E_n|$ grows exponentially. For any Borel subset $W$ of $\PMF(S)$, denote by $Sect_{W}$ the union of geodesics starting from $o$ and ending at $W$.  We first recall the following theorem in \cite{ABEM} in our setting. Let $$C(n,\rho) = \left\{\gamma \in \M(S): d_T(\gamma \cdot o,o) \in (n-\rho,n+\rho)\right\}.$$

\begin{theorem}[\cite{ABEM},Theorem 2.10]\label{ABEM2}
	Let $W$ and $V$ be two Borel subsets of $\PMF(S)$ with measure zero boundaries. Then as $R$ tends to $\infty$,
	\begin{displaymath}
	\begin{aligned}
	& \left|\{\gamma \in C(n,\rho): \gamma \cdot o \in Sect_{W} ~ \mbox{and} ~\gamma^{-1} \cdot o \in Sect_{V} \}\right|\\
	&\phantom{=\;} \asymp K e^{h n} \nu(W)\nu(V).
	\end{aligned}
	\end{displaymath}
	where $K$ is a constant depending on $g$, $\rho$ and $o$. In fact, using the notations in \cite{ABEM}, one has $K = \frac{2sinh(h\rho)\|\nu(\PMF(S))\|^2}{hm(\mathcal{M}_g)},$ where $m(\mathcal{M}_g)$ is the push forward of the Masur-Veech volume.
\end{theorem}
\begin{corollary}\label{corollary:exponentialgrowth}
	Let $n \gg 0 $ and $K$ be the constant in Theorem \ref{ABEM2}. Then $|E_n| \asymp Ke^{hn}$ (up to a subsequence). In particular, $\lim_{n \rightarrow \infty}|E_n| = \infty$. 
\end{corollary}
\begin{proof}
	As $E_n \subset C(n,\rho)$ and, by Theorem \ref{ABEM2}, $|C(n,\rho)| \asymp Ke^{hn}$, it is obvious that $|E_n| ~\widetilde{<}~ Ke^{hn}$. We now show that $ |E_n| ~\widetilde{>}~ Ke^{hn}$. Recall that $U(\theta,\epsilon) = \cup_{R> 0}U(R,\theta,\epsilon)$ with $\nu(U(\theta,\epsilon)) = 1$ and $U(S,\theta,\epsilon) \subset U(T,\theta,\epsilon)$ for $T > S$. Let $\delta_1 > 0$ small enough and choose $R \gg 0$ such that $$1-\delta_1 \leq \nu(U(R, \theta,\epsilon)) \leq 1, ~\nu(\partial U(R,\theta,\epsilon)) = 0.$$ By Theorem \ref{ABEM2} again, for any $\delta_2 > 0$ small enough, there exists $N(\delta_2)$ so that, whenever $n \geq N(\delta_2)$ and $\frac{1}{3h}\ln \ln n > R$, 
	\begin{displaymath}
	\begin{aligned}
	&\left|\{\gamma \in C(n,\rho): \gamma \cdot o \in Sect_{U(R,\theta,\epsilon)} ~ \mbox{and} ~\gamma^{-1} \cdot o \in Sect_{U(R,\theta,\epsilon)} \}\right|\\
	& \geq Ke^{-\delta_2} e^{hn} \left(\nu(U(R,\theta,\epsilon))\right)^2\\
	& \geq Ke^{-\delta_2}(1-\delta_1)^2e^{hn}.
	\end{aligned}
	\end{displaymath}
	On the other hand, by the choice of $n$ and the definition of $U(R, \theta,\epsilon)$, 
	
	\begin{displaymath}
	\begin{aligned}
	&\{\gamma \in C(n,\rho): \gamma \cdot o \in Sect_{U(R,\theta,\epsilon)} ~ \mbox{and} ~\gamma^{-1} \cdot o \in Sect_{U(R,\theta,\epsilon)} \} \\
	& \subset E_n.
	\end{aligned}
	\end{displaymath}
	Therefore, we have	$|E_n| \geq Ke^{-\delta_2}(1-\delta_1)^2e^{hn}$. As $\delta_1$ and $\delta_2$ can be arbitrary small, one has $|E_n| ~\widetilde{>}~ Ke^{hn}$.
\end{proof}

\begin{corollary}\label{weakconvergence}
	For all Borel subsets $W,V \subset \PMF(S)$ such that $\nu(\partial W) = \nu(\partial V) = 0$, $$\limsup_{n \rightarrow \infty}\frac{1}{|E_n|}|\left\{\gamma \in E_n: Pr(\gamma^{-1})\in W ~\mbox{and}~Pr(\gamma)\in V \right\}| \leq \nu(W)\nu(V).$$
\end{corollary}
\begin{proof}
	By Corollary \ref{corollary:exponentialgrowth} and Theorem \ref{ABEM2}, $|E_n| \asymp |C(n,\rho)|$. Notice that $Pr(\gamma) \in V$ if and only if $\gamma \cdot o \in Sect_V$. Hence,
	\begin{displaymath}
	\begin{aligned}
	&\limsup_{n \rightarrow \infty}\frac{1}{|E_n|}\left|\left\{\gamma \in E_n: Pr(\gamma^{-1})\in W ~\mbox{and}~Pr(\gamma)\in V \right\}\right|\\
	&\leq \limsup_{n \rightarrow \infty}\frac{1}{|E_n|}\left|\left\{\gamma \in C(n,\rho): Pr(\gamma^{-1})\in W ~\mbox{and}~Pr(\gamma)\in V \right\}\right|\\
	&= \limsup_{n \rightarrow \infty}\frac{|C(n,\rho)|}{|E_n|}\frac{1}{|C(n,\rho)|}\left|\left\{\gamma \in C(n,\rho): Pr(\gamma^{-1})\in W ~\mbox{and}~Pr(\gamma)\in V \right\}\right|\\
	&\leq \nu(W)\nu(V).
	\end{aligned}
	\end{displaymath}
	
\end{proof}

\subsection{Shadow lemma.}
The following Shadow lemma will be used in the proof of uniform boundedness (Section \ref{subsection:uniformboundedness}).
\begin{lemma}[\cite{Yang}, Lemma 6.3]\label{lemma:Shodowlemma}
	There exists $R_0 > 0$ such that for every $R \geq  R_0$, there is a real number $C \geq 1$ depending on $R$, such that for all $n \gg \rho$ and $g \in E_n$,  $$  \frac{1}{C} \exp{(-h d(o,g \cdot o))} \leq \nu(Pr(B(g \cdot o,R))) \leq C\exp{(-h d(o, g\cdot o))}. $$ 

\end{lemma}

Recall that $U$ has a full measure. We need a lemma that relates Busemann functions to extremal lengths. Recall that if $(X,d_{X})$ is a metric space and $\xi$ is a geodesic ray starting from a point $x_0 \in X$, then {\it the Busemann function} associated to the geodesic ray $\xi$ is the function $b_{\xi}$ on $X$ defined by $$b_{\xi}: x \mapsto \lim_{t \rightarrow \infty}\left(d_{X}(x, \xi(t))-t\right).$$  For $(X = \T(S),d_{X}=d)$ and a geodesic ray $\xi$ starting from $o$, one has
\begin{lemma}[\cite{Walsh09}, \cite{su2018horospheres}]\label{busemann}
	If $[\xi]$ is uniquely ergodic, then the Busemann function associated to the geodesic ray in the direction $[\xi]$ is $$\forall x \in \T(S),  b_{[\xi]}(x) = \frac{1}{2}\ln \left(\frac{\E_x(\xi)}{\E_{o}(\xi)}\right).$$
\end{lemma}

\section{Harish-Chandra estimates}\label{Section:Harish-Chandraestimations}
This section is devoted to prove the following Harish-Chandra estimates.
\begin{theorem}\label{Harish-Chandra}
	Given $n \gg \rho $. There exist $a_1 > 0, a_2 > 0,b_1, b_2,c_1 >0$ depending on $\epsilon,o, g,\theta,\rho$ such that $$ \forall \gamma \in E_n,~
	(a_1n-c_1\ln \ln n +b_1)e^{-\frac{h}{2}n} \leq \Phi(\gamma) \leq (a_2n+b_2)e^{-\frac{h}{2}n}. $$
\end{theorem}

Recall that
\begin{displaymath}
\begin{aligned}
& \Phi(\gamma) = <\pi_{\nu}(\gamma)\mathds{1}_{\PMF(S)}, \mathds{1}_{\PMF(S)}>_{L^2(\PMF(S),\nu)}\\
&\phantom{=\;\;} = \int_{\PMF(S)}\left(\frac{\E_{o}(\xi)}{\E_{\gamma.o}(\xi)}\right)^{\frac{h}{4}}d\nu([\xi]).
\end{aligned}
\end{displaymath}
\begin{remark}
	\begin{itemize}
		\item[1.] In \cite{Boyer17}, the left side is of the form $(an+b)e^{-\frac{\alpha}{2}n}$. However, it seems that some other terms like $\ln \ln n$ should be added for mapping class groups if we require $\lim_n \frac{|E_n|}{C(n,\rho)} = 1$.
		\item[2.] The following oberservation will be useful, namely $\Phi(\gamma) \asymp ne^{-\frac{hn}{2}}$.
	\end{itemize}	
	
\end{remark}

The proof is divided into several lemmas and will be given at the end of this section.

\subsection{Reduction to intersection numbers.}
By our convention, for every $\xi \in \PMF(S)$, one has $\E_{o}(\xi) = 1$, we then have $$\Phi(\gamma) = \int_{\PMF(S)}\left(\frac{1}{\E_{\gamma.o}(\zeta)}\right)^{\frac{h}{4}}d\nu(\zeta).$$
\noindent Let $\xi_{\gamma}$ be the direction of $[o,\gamma \cdot o]$. In order to estimate $\Phi(\gamma)$, we will relate it to the following integrations on intersection numbers:
$$\Psi(\gamma) = \int_{\PMF(S)}\left(\frac{1}{i(\xi_{\gamma},\eta)}\right)^{\frac{h}{2}}d\nu(\eta).$$

Denote $$\Psi(\gamma)_{\geq A} = \int_{\{\eta \in \PMF(S): i(\xi_{\gamma}, \eta) \geq A\}}\left(\frac{1}{i(\xi_{\gamma},\eta)}\right)^{\frac{h}{2}}d\nu(\eta).$$

The first step is to bound $\Phi(\gamma)$ from above. This can be easily done by Minsky's inequality. Namely,
\begin{lemma}[Minsky's inequality \cite{Minsky93}]\label{Minskyinequailty}
	Let $\xi$ and $\eta$ be two measured foliations on $S$ and $x \in \T(S)$, then $$i^2(\xi, \eta) \leq \E_{x}(\xi) \E_{x}(\eta),$$ where the equilty holds if and only if there is a qudratic differential $q$ so that the vertical measured foliation of $q$ on $X$ is $\xi$ and the horizontal measured foliation is $\eta$.
\end{lemma}
\begin{corollary}\label{corollary:upperbound}
	There exist constants $C_3 = C_3(g,\rho) > 0$ and $C_4 = C_4(g,\rho) > 0$ such that, for every $M \in (0,1)$ and every $\gamma \in \M(S)$, $$\Phi(\gamma) \leq C_3e^{-\frac{h}{2}n}\Psi(\gamma)_{\geq M} + C_4e^{\frac{h}{2}n} \nu(\{\eta \in \PMF(S): i(\eta, \xi_{\gamma}) \leq M\} ).$$	
\end{corollary}
\begin{proof}
	Decompose $\PMF(S)$ into two subsets $A = \{ \eta \in \PMF(S): i(\eta, \xi_{\gamma}) \leq M \}$ and $B = \{ \eta \in \PMF(S): i(\eta, \xi_{\gamma}) \geq M \}$. Then we have
	\begin{equation}
	\begin{aligned}
	&\Phi(\gamma) = \underbrace{\int_{A}\left(\frac{1}{\E_{\gamma \cdot o}(\eta)}\right)^{\frac{h}{4}}d\nu(\eta)}_{I} + \underbrace{\int_{B}\left(\frac{1}{\E_{\gamma \cdot o}(\eta)}\right)^{\frac{h}{4}}d\nu(\eta)}_{II}\\
	&= {\rm I + II}.
	\end{aligned}
	\end{equation}
	By Kerckhoff's formula, ${\rm I} \prec_{g,\rho} e^{\frac{h}{2}n} \nu(\{\eta \in \PMF(S): i(\eta, \xi_{\gamma}) \leq M\} )$. Thanks to Lemma \ref{busemann}, we can replace $\E_{\gamma \cdot o}(\xi_{\gamma})$ in Lemma \ref{Minskyinequailty} by $e^{-2n}$, so one has $$\frac{1}{i^2(\xi_{\gamma}, \eta) e^{2n}} \succ_{g,\rho} \frac{1}{\E_{\gamma \cdot o}(\eta)},$$ which gives the bound for the term ${\rm II}$.
\end{proof} 

In order to bound $\Phi(\gamma)$ from below, we will use the fact that $\gamma \in E_n$. 
\begin{lemma}\label{lemma:inverseminsky}
	There exists a constant $F$ depending on $g,o,\epsilon,\theta,\rho$ such that if $i(\xi_{\gamma}, \eta) \geq F \ln ne^{-2n}$, where $\eta \in U(\epsilon,\theta)$ and $\gamma \in E_n$, then $i^2(\xi_{\gamma},\eta) \succ_{g,o,\epsilon,\theta,\rho} \E_{\gamma \cdot o}(\eta)e^{-2n}$. 
\end{lemma}
\begin{proof}
	First we remark that, since both $\eta$ and $\xi_{\gamma}$ are uniquely ergodic, by [\cite{klarreich2018boundary}, Proposition 5.1], there is a geodesic whose horizontal and vertical measured foliations are in the projective classes $\xi_{\gamma}$ and $\eta$ respectively. Hence we have a geodesic triangle $\triangle(o,\xi_{\gamma}, \eta)$. As $\gamma \in E_n$, Lemma \ref{lemma:longsegment} implies that there is a geodesic segment $I$ of length $\ell = \frac{1}{3h}\ln \ln n$ in $[o,\gamma \cdot o]$ ending at $\gamma \cdot o$ that has at least proportion $\theta$ in $\T_{\epsilon}(S)$. By Theorem \ref{theorem:DDM} (although the theorem is for finite triangles, it is easy to see that it can be applied to $\triangle(o,\xi_{\gamma}, \eta)$), $$I \cap \mathcal{N}_{D}([\eta,\xi_{\gamma}]\cap[o,\eta]) \ne \emptyset,$$ where $D$ comes from Theorem \ref{theorem:DDM}. Choose $q \in I \cap \mathcal{N}_{D}([\eta,\xi_{\gamma}]\cap[o,\eta])$. Then there are two possibilities:\\
	
	\noindent Case 1: $d(q, y) \leq D$ with $y\in [\xi_{\gamma},\eta]$.\\
	
	Then we have, by Kerckhoff's formula and Lemma \ref{Minskyinequailty},
	\begin{displaymath}
	\begin{aligned}
	& i^2(\xi_{\gamma},\eta) = \E_{y}(\eta)\E_{y}(\xi_{\gamma})\\
	&\succ_{g,o,\theta,\epsilon} \E_{q}(\eta)\E_{q}(\xi_{\gamma})\\
	&=\E_{q}(\eta)e^{-2d(o,q)}\\
	&\geq \E_{\gamma \cdot o}(\eta)e^{-2n},
	\end{aligned}
	\end{displaymath}
	which means that, in this case, we always have $i^2(\xi_{\gamma},\eta) \succ_{g,o,\epsilon,\theta} \E_{\gamma \cdot o}(\eta)e^{-2n}$.\\
	
	\noindent Case 2: $d(q, y) \leq D$ with $y\in [o,\eta]$.\\
	
	Then we have
	
	\begin{displaymath}
	\begin{aligned}
	&i^2(\xi_{\gamma},\eta) \leq \E_{y}(\eta)\E_{y}(\xi_{\gamma})\\
	&\sim_{g,o,\theta,\epsilon} \E_{y}(\eta)\E_{q}(\xi_{\gamma})\\
	&\sim_{g,o,\theta,\epsilon,\rho} e^{-4d(o,q)} = e^{-4(d(o,\gamma \cdot o)-d(q,\gamma \cdot o))}\\
	&\leq e^{-4n}e^{4\ell}.
	\end{aligned}
	\end{displaymath}
	
	Therefore, in this case we have a constant $F_1$ depending on $g,o,\epsilon,\theta,\rho$ such that $$i(\xi_{\gamma},\eta) \leq F_1e^{-2n}e^{2\ell} \leq F_1e^{-2n}e^{\ln \ln n} = F_1\ln n e^{-2n}.$$ Thus if we take $F \gg F_1$ and require $i(\xi_{\gamma},\eta) \geq F\ln ne^{-2n}$, it forces us in the Case 1 which implies the conclusion that $$i^2(\xi_{\gamma},\eta) \succ_{g,o,\epsilon,\theta} \E_{\gamma \cdot o}(\eta)e^{-2n}.$$
\end{proof}

\begin{corollary}\label{corollary:lowerbound}
	For every $\gamma \in E_n$, take $\bar{M} = F\ln ne^{-2n}$ where $F$ is the constant in Lemma \ref{lemma:inverseminsky}. Then $\Phi(\gamma) \succ_{g,o,\epsilon,\theta,\rho} e^{-\frac{h}{2}n}\Psi(\gamma)_{\geq \bar{M}}.$
\end{corollary}
\begin{proof}
	Note that $U(\epsilon,\theta)$ has full measure. Hence, by Lemma \ref{lemma:inverseminsky}
	\begin{displaymath}
	\begin{aligned}
	&\Phi(\gamma) = \int_{U(\epsilon,\theta)}\left(\frac{1}{\E_{\gamma.o}(\eta)}\right)^{\frac{h}{4}}d\nu(\eta)\\
	&\geq \int_{\{\eta \in U(\epsilon,\theta): i(\eta,\xi_{\gamma}) \geq \bar{M}\}}\left(\frac{1}{\E_{\gamma.o}(\eta)}\right)^{\frac{h}{4}}d\nu(\eta)\\
	&\succ_{g,o,\epsilon,\theta} e^{-\frac{hn}{2}}\Psi(\gamma)_{\geq \bar{M}}.
	\end{aligned}
	\end{displaymath}
\end{proof}
\begin{lemma}\label{lemma:log}
	Assume that there exist a sequence $\{\xi_k\} \in \PMF(S)$ converges to $\xi_{\gamma}$ and constants $N_0 > 0, a > 0, b > 0$ such that for every $k > 0$, $$\forall N \leq N_0, ~~~a N^{\frac{h}{2}} \leq \nu(\{\eta \in \PMF(S): i(\eta, \xi_k) \leq N\}) \leq b N^{\frac{h}{2}} ,$$ then there exist $N_0, A > 0,B > 0 ,D_1,D_2$ such that $$\forall N \leq N_0,~~ -A\ln N + D_1 \leq \Psi(\gamma)_{\geq N }\leq -B\ln N +D_2. $$
\end{lemma}
\begin{proof}
Fix $N \leq N_0$, then the set $\{\eta \in \PMF(S): i(\eta,\xi_{\gamma}) \geq N\}$ is compact. Since $i(\xi_k,\cdot)$ converges to $i(\xi_{\gamma},\cdot)$ uniformly on compact sets outside $\{\xi_{\gamma}\}$, there exists $K_1$ so that, when $k \geq K_1$, $$\{\eta: i(\xi_k, \eta) \geq 2N\} \subset \{\eta: i(\xi_{\gamma},\eta) \geq N\}.$$ Hence $$\forall k \geq K_1, \int_{\{\eta \in \PMF(S): i(\xi_{k},\eta) \geq 2N\}}\left(\frac{1}{i(\xi_{\gamma},\eta)}\right)^{\frac{h}{2}}d\nu \leq \int_{\{\eta \in \PMF(S): i(\xi_{\gamma},\eta) \geq N\}}\left(\frac{1}{i(\xi_{\gamma},\eta)}\right)^{\frac{h}{2}}d\nu.$$ By convergence uniform on compact sets again, there is $K_2$ such that when $k \geq K_2$, for $\eta \in \{\eta: i(\eta,\xi_(\gamma)) \geq N\}$, $$\frac{1}{2} \leq \left(\frac{i(\xi_k,\eta)}{i(\xi_{\gamma},\eta)}\right)^{\frac{h}{2}} \leq 2.$$ Take $k > \max\{K_1,K_2\}$, then 
\begin{equation*}
    \begin{aligned}
        &\frac{1}{2}\int_{\{\eta \in \PMF(S): i(\xi_{k},\eta) \geq 2N\}}\left(\frac{1}{i(\xi_{k},\eta)}\right)^{\frac{h}{2}}d\nu\\
        &\leq \int_{\{\eta \in \PMF(S): i(\xi_{k},\eta) \geq 2N\}}\left(\frac{1}{i(\xi_{\gamma},\eta)}\right)^{\frac{h}{2}}d\nu \\
        &\leq \Psi(\gamma)_{\geq N}
    \end{aligned}
\end{equation*}
Since for $\eta \in \PMF(S)$, $$\lim_{k \to \infty} \mathds{1}_{\{\eta \in \PMF(S): i(\xi_{k},\eta) \geq N\}}(\eta)\left(\frac{1}{i(\xi_{k},\eta)}\right)^{\frac{h}{2}} = \mathds{1}_{\{\eta \in \PMF(S): i(\xi_{\gamma},\eta) \geq N\}}(\eta)\left(\frac{1}{i(\xi_{\gamma},\eta)}\right)^{\frac{h}{2}},$$
by Fatou's lemma, there is $K_3$ such that when $k \geq K_3$, $$\Psi(\gamma)_{\geq N} \leq \int_{\{\eta \in \PMF(S): i(\xi_{k},\eta) \geq N\}}\left(\frac{1}{i(\xi_{k},\eta)}\right)^{\frac{h}{2}}d\nu + 1.$$
Define $$\Psi(k)_{\geq N} = \int_{\{\eta \in \PMF(S): i(\xi_{k}, \eta) \geq N\}}\left( \frac{1}{i(\xi_{k},\eta)}\right)^{\frac{h}{2}}d\nu(\eta),$$ we have for $k >> 0,$ $$\frac{1}{2}\Psi(k)_{\geq 2N} \leq \Psi(\gamma)_{\geq N} \leq \Psi(k)_{\geq N} + 1.$$ Hence it is sufficient to show that
 there exist $A,B,D_1,D_2$ such that $$-A\ln N + D_1 \leq \Psi(k)_{\geq N }\leq -B\ln N +D_2.$$
Now the proof is close to part of \cite[Proposition 3.2]{Boyer17}. We repeat here for completeness. Namely,
	\begin{equation}
	\begin{aligned}
	& \Psi(k)_{\geq N}= \int_{\{\eta \in \PMF(S): i(\xi_{k}, \eta) \geq N\}}\left( \frac{1}{i(\xi_{k},\eta)}\right)^{\frac{h}{2}}d\nu(\eta)\\
	&=\int_{}\nu\left(\left\{\eta \in \PMF(S): \left(\frac{1}{i(\xi_{k},\eta)}\right)^{\frac{h}{2}} \geq t\right\}\right)dt\\
	&=\int_{1}^{\frac{1}{N^{\frac{h}{2}}}}\nu\left(\left\{\eta \in \PMF(S): i(\xi_{k},\eta)\leq \frac{1}{t^{\frac{2}{h}}}\right\}\right)dt\\
	&=\int_{1}^{N_0}\nu\left(\left\{\eta \in \PMF(S): i(\xi_{k},\eta)\leq \frac{1}{t^{\frac{2}{h}}}\right\}\right)dt \\
	&+\int_{N_0}^{\frac{1}{N^{\frac{h}{2}}}}\nu\left(\left\{\eta \in \PMF(S): i(\xi_{k},\eta)\leq \frac{1}{t^{\frac{2}{h}}}\right\}\right)dt.
	\end{aligned}
	\end{equation}
	\noindent By the assumption and $\nu$ is a probability measure, one can have the conclusion.
\end{proof}
 We include the following observation. 
\begin{lemma}\label{lemma:approximation}
	Let $\beta \in U(\theta,\epsilon)$. If there exists $N_0 > 0$ small enough and a sequence of points $\{\xi_k(\beta) \in \PMF(S)\}$ converges to $\eta$ so that, for each $k$, $$\forall N \leq N_0,~  \nu(\{\eta \in \PMF(S):i(\eta,\xi_k(\beta)) \leq N\}) \leq bN^{\frac{h}{2}},$$ then there exist $N_0 > 0,a'>0$ such that $$\forall N \leq N_0,~  \nu(\{\eta \in \PMF(S):i(\eta,\beta)\leq N\}) \leq a'N^{\frac{h}{2}}.$$ 
\end{lemma}
\begin{proof}
Denote $A_k(N) = \{\eta: i(\eta, \xi_k(\beta)) \leq N\}$ and $B=\{\eta: i(\eta, \beta) \leq N\}$. Then since $i(\cdot,\cdot)$ is continuous, so $$\lim_{k \to \infty}\mathds{1}_{A_k(N)}(x)= \mathds{1}_{B(N)}(x).$$ By Fatou's lemma, $$\nu(B(N)) \leq \liminf_{k} \nu(A_k(N)) \leq bN^{\frac{h}{2}}. $$ 
\end{proof}

\subsection{A basic example.}
 Before continuing our discussions, we digress to the case of once-punctured torus $S_{1,1}$. Some standard facts are taken from [\cite{Miyachi17}, 7.2 Examples].\\

Let $S = S_{1,1}$. Then $\M(S) = SL(2,\mathbb{Z})$ and $\T(S) = \mathbb{H}^{2}$, the upper half plane. Take $o$ to be $i \in \mathbb{H}^{2}$. The space $\MF(S)$ of measured foliations can be identified with the real plane module the inversion, namely $\{\mathbb{R}^2-(0,0)\}/\{I,-I\}$. By the ergodicity of the Thurston measure $\nu_{Th}$, up to a constant multiple, the measure $\nu_{Th}$, which is defined to be the weak limit of counting measures on $\MF(S)$, can be identified with the Lebesgue measure on $\mathbb{R}^2$. Rays in $\{\mathbb{R}^2-(0,0)\}/\{I,-I\}\}$ are then identified with points in $\PMF(S)$. It implies that $\PMF(S)$ can be identified with $\mathbb{R}P^1$. Notice that all identifications here are $\M(S)-$equivariant. Hence $\PMF(S)$ can be represented as $\{[x:y]: x^2+y^2 \neq 0, x,y \in \mathbb{R}\}$, or $\mathbb{R}\cup \{\infty\}$. $\bar{\T(S)}$ is then the usual compactification of $\mathbb{H}^{2}$. In this case, $\M(S)$ acts on $\bar{\T(S)}$ via linear fractional transformations. For $(x,y) \in \mathbb{R}^2$, the extremal length at $o$ is $$ \E_{o}((x,y)) = x^2+y^2,$$ hence the image of $\PMF(S)$ under $\tau$ is the circle. We will ignore the difference between $\mathbb{R}^2$ and $\mathbb{R}^2/\{I,-I\}$. For two points $(x,y),(p,q) \in \MF(S)$, the intersection number is $|qx-py|$. Write the image of $\PMF(S)$ in the form of $(\sin(\theta), \cos(\theta))$, and fix any $\xi = (\sin(\theta_0),\cos(\theta_0)) \in \PMF(S).$ Let $M$ to be small enough, then
\begin{equation}
\begin{aligned}
&\{\eta \in \PMF(S): i(\xi, \eta) \leq M\} \\
&=\{\theta \in [0,2\pi]: |\sin(\theta)\cos(\theta_0)- \cos(\theta)\sin(\theta_0)| \leq M\}\\
&=\{\theta \in [0,2\pi]: |\sin(\theta-\theta_0)| \leq M \}\\
&=\{\theta \in [0,2\pi]: -M \leq \sin(\theta-\theta_0) \leq M \}.
\end{aligned}
\end{equation}
As $M$ is small enough, $\sin(\theta)$ is almost the same as $\theta$, so there exist constants $A$ and $B$, so that $$AM \leq \nu(\{\eta \in \PMF(S): i(\xi, \eta) \leq M\}) \leq BM,$$
Notice that, when $S$ is $S_{1,1}$, we have $h =6g-6+n= 6\times 1- 6+ 2 \times 1 = 2$, hence $\frac{h}{2} = 1$.

\subsection{Approximation by pant curves.}
In this subsection, we will prove the following proposition which is an analogue of Diophantine approximation.
\begin{proposition} \label{proposition:pantscurveapp}
There are $A > 0,B >0$, depending on $\epsilon$ and $o$, such that, for every $\eta \in U(\theta,\epsilon)$, there is a sequence $\{\xi_k = \xi_k(\eta)  \in \PMF(S)\}^{\infty}_{k = 1}$ satisfying 
\begin{itemize}
	\item[($\star_1$)] For each $k$,  $\xi_k = [x_k = \sum_{i=1}^{3g-3}\alpha^k_i]$, where $\{\alpha^k_i \}^{3g-3}_{i=1}$ is a pants decomposition of $S$.
	\item[($\star_2$)] For each $k$ and $1 \leq i \leq 3g-3$, there is a sequence $\{t_k\}$ of positive reals such that  $Ae^{2t_k} \leq \E_{o}(\alpha^k_i) \leq Be^{2t_k}$, $Ae^{2t_k} \leq \E_{o}(x_k) \leq Be^{2t_k}$ and $\lim_{k \to \infty}t_k = \infty$.
	\item[($\star_3$)] The limit of $\{\xi_k=\frac{x_k}{\sqrt{\E_{o}(x_k)}}\}$ in $\PMF(S)$ is $\eta$. 
\end{itemize}
\end{proposition}

Any pants curve $x_k$ in Proposition \ref{proposition:pantscurveapp} will be called \textit{a thick pants curve} and the sequence $\{\xi_k(\eta)\}$ will be called \textit{pants curves associated to $\eta$}. This proposition together with the next subsection shows the assumption in Lemma \ref{lemma:log} holds. We prove this proposition by using the systole map which has been considered, for instance, in \cite{MasurMinsky99}. \\

Recall that, thanks to Bers' theorem, there is a constant $C_1=C(g)$, depending only on the genus $g$, such that for every point $x \in \T(S)$, there exists a pants decomposition, namely a collection of $3g-3$ essential simple closed curves $\mathcal{P} = \{\alpha_1,\cdots,\alpha_{3g-3}\}$, such that $$\forall 1\leq i \leq 3g-3,~ \E_{x}(\alpha_i) \leq C_1^2.$$ If $x$ is further assumed to be in $\T_{\epsilon}(S)$, one can choose a collection of $3g-3$ essential simple closed curves $ \{\alpha_1(x),\cdots,\alpha_{3g-3}(x)\} $ on $S$ such that $$\forall 1\leq i \leq 3g-3, \epsilon \leq \E_{x}(\alpha_i(x)) \leq C^{2}_{1}.$$ Let $\beta_x$ be the measured foliation $\beta_x = \sum_{i=1}^{3g-3}\alpha_i(x)$ and $[\beta_x]$ be its projective class in $\PMF(S)$. By Hubbard-Masur, there is a unique unit holomorphic quadratic differential $q = q(\beta_x)$ on $o \in \T(S)$ whose projective class of the vertical measured foliation is $[\beta_x]$. \\

Now given $\eta \in U(\theta,\epsilon)$, by the construction of $U(\theta,\epsilon)$ (cf. the argument after Theorem \ref{theorem:DowDucMas}), the geodesic ray $g_t = [o, \eta)$ cannot leave $\T_{\epsilon}(S)$ eventually. Hence, there is a sequence of points $\{y_k=y_k(\eta) \in \T_{\epsilon}(S)\}$ in $g_t$ tends to $\eta$ in $\bar{\T(S)}$. By the above discussion, we obtain a sequence of pants curves $\{\beta_{y_k}\}$ and a sequence of points $\{[\beta_{y_k}]\}$ in $\PMF(S)$. \\


Let $y$ be any $y_k(\eta)$ and denote $\beta_y = \sum \alpha_i(y)$. Let $g_t$ be the Teichm\"{u}ller geodesic ray starting from $o$ in the direction $q(\beta_y)$. Along the geodesic ray $[o, [\beta_y])$, the moduli of the core curves $\alpha_i(y)$ is multiplied by $e^{2t}$, so by the geometric definition of the extremal length \cite[Page 32]{Kerckhoff80}, $$\E_{g_t}(\alpha_i(y)) \leq e^{-2t}\E_o(\alpha_i(y)).$$ Hence along the geodesic ray in direction $[\beta_y]$, there is a point $t_y$ in the geodesic ray such that $t_y$ has maximal distance with $o$ and $$\forall 1\leq i \leq 3g-3,~ \epsilon \leq \E_{t_y}(\alpha_{i}(y)) \leq C^{2}_{1} .$$

\begin{lemma}\label{lemma:boundeddistance}
	There is a constant $C_2$ depending on $g$ and $\epsilon$ such that $$ d(y, t_y) \leq C_2.$$ 	
\end{lemma}
The proof is based on the following theorem. The theorem is used in the form of \cite[Theorem 5.3]{ABEM}. For the definition of twist numbers $tw(\alpha,\beta)$, the reader is referred to \cite{PennerHarer}.
\begin{theorem}[\cite{Minsky96}]\label{Minskyextremal}
	Let $x \in \T(S)$ and $\mathcal{P} = \{\alpha_1, \cdots, \alpha_{3g-3}\}$ be a pants decomposition produced by the Bers' theorem	mentioned above. Then for any simple closed curve $\beta$, 
	\begin{equation}\label{equation:extremal}
	\E_{x}(\beta) \sim_g \max_{1 \leq i \leq 3g-3}\left( \frac{i^2(\beta, \alpha_i)}{\E_{x}(\alpha_i)} + tw^2(\beta,\alpha_i)\E_{x}(\alpha_i) \right).	
	\end{equation}
\end{theorem}
\begin{proof}[Proof of Lemma \ref{lemma:boundeddistance}]
	By Kerckhoff's formula, we only need to bound the ratio $\frac{\E_{y}(\beta)}{\E_{ t_y}(\beta)}$ for any essential simple closed curve $\beta$ on $S$. However, by the construction of $\{\alpha_i(y)\}$, the two hyperbolic surfaces $t_y$ and $y$ have the same pants decomposition$ \{ \alpha_1(y), \cdots, \alpha_{3g-3}(y)\}$ satisfying the condition in Theorem \ref{Minskyextremal}.  Since, for $1 \leq i \leq 3g-3$, both $\E_{y}(\alpha_i)$ and $\E_{t_y}(\alpha_i)$ are bounded below by the constant $\epsilon$ and above by a constant $C^{2}_{1}$ depending only on $g$, one can complete the proof of the lemma by using  Equation (\ref{equation:extremal}).
\end{proof}

Using Theorem \ref{Minskyextremal} again, $\E_{t_y}(\beta_y) \sim_{g,\epsilon} \E_{y}(\beta_y) \sim_{g,\epsilon} 1$. Since $[\beta_y]$ is the direction from $o$ to $t_y$, $\E_{t_y}(\beta_y) = e^{-2t}\E_{o}(\beta_y)$, hence $\E_{o}(\beta_y) \sim_{g,\epsilon} e^{2t}$ where $t = d(o,t_y)$. By Lemma \ref{lemma:boundeddistance} and triangle inequality, $t \sim_{g,\epsilon} d(o,y)$.
\begin{proof}[Proof of Proposition \ref{proposition:pantscurveapp}]
According to above discussions, for each $\eta \in U(\theta,\epsilon)$, we have a sequence $\{y_k = y_k(\eta) \in \T(S)\}$ and for each $y_k$, we have $\beta_{y_k} = \sum_i \alpha_i(y_k)$. For each $k$, $\{\alpha_i(y_k)\}^{i= 3g-3}_{i=1}$ is a pants decomposition of $S$ and by the discussion above, $\E_o(\beta_{y_k}) \sim_{g,\epsilon} e^{2d(o,y_k)}$. We first show that the sequence $\{\frac{\beta_{y_k}}{\E_o(\beta_{y_k})}\}$ converges to $\eta$.  Indeed, since $\eta$ is uniquely ergodic, hence by \cite[Proposition 2.4]{HamenstadtStability}, the projection of $[o,\eta)$ to the curve graph is unbounded and the projection of the sequence $y_k$, up to a subsequence, converges in the curve graph to $\eta$. By \cite[Theorem 1.4]{klarreich2018boundary}, $\beta_{y_k}$, up to a subsequence, converges in $\PMF(S)$ to $\eta$. Notice that, the fact that $\{\frac{\beta_{y_k}}{\E_o(\beta_{y_k})}\}$ converges to $\eta$ also implies that for eack $i$, $\{\frac{\alpha_i(y)}{\E_o{\alpha_i(y)}}\}$ converges to $\eta$ as well since $\eta$ is uniquely ergodic. We have to verify $(\star_2)$ in Proposition \ref{proposition:pantscurveapp}. If there is a subsequence, still denote by $\{y_k\}$, of $\{y_k\}$ such that for each $k$ and each $1 \leq i \leq 3g-3$, $\E_o(\alpha_i(y_k)) \sim_{g,\epsilon} e^{2d(o,y_k)}$, then we are done. Otherwise, since $\E_o(\beta_{y_k}) \sim_{g,\epsilon} e^{2d(o,y_k)}$ and $\E_o(\beta_{y_k}) = \sum_i \E_o(\alpha_i(y_k))$, for each $k$, there is $i_k \in \{1,2,\cdots,3g-3\}$ such that $\E_o(\alpha_{i_k}(y_k)) \sim_{g,\epsilon} e^{2d(o,y_k)}$. For each $k$, we are now going to modify the pants decomposition $\{\alpha_j(y_k)\}$ as follows. Write $\beta_{y_k} = \alpha_1(y_k) + \alpha_2(y_k) + \cdots + \alpha_{m}(y_k) + \cdots + \alpha_{3g-3}(y_k)$ such that the first $m$ terms has extremal length less than $e^{2d(o,y_k)}$ but the last $3g-3-m$ terms are comparable with $e^{2d(o,y_k)}$ (the constants involved only depends on $o,\epsilon$). So $i_k > m$. For $j = 1$, find another curve $\alpha'_j(y_k)$ so that $\alpha'_{j}(y_k)$ intersects only with $\alpha_1(y_k)$ among the pants curve $\{\alpha_i(y_k)\}$ and has extremal length $\E_o(\alpha'_{j}(y_k))$ comparable with $e^{2d(o,y_k)}$.  The simple closed curve $\alpha'_{j}(y_k)$ can be constructed by enough many twists since $\E_o(\alpha_1(y_k))$ less than $e^{2d(o,y_k)}$. Then for $j = 2$,  find another curve $\alpha'_j(y_k)$ so that $\alpha'_{j}(y_k)$ intersects only with $\alpha_2(y_k)$ among the pants curve $\{\alpha'_1(y_k), \alpha_2(y_k), \cdots, \alpha_{3g-3}(y_k)\}$ and has extremal length $\E_o(\alpha'_{j}(y_k))$ comparable with $e^{2d(o,y_k)}$. We continue this process for $m$ terms. Finally we have, for each $k$, a pants decomposition $\{\alpha'_1(y_k),\cdots,\alpha'_{m}(y_k),\alpha_{m+1},\cdots,\alpha_{3g-3}(y_k)\}$. Denote this decomposition by $\{\delta_i(y_k)\}$ and $\delta(y_k) = \sum_i \delta_i(y_k)$. On the one hand, by the construction and the fact that $\E_o(\delta(y_k)) = \sum_i \E_o( \delta_i(y_k))$, the condition $(\star_2)$ in Proposition \ref{proposition:pantscurveapp} holds. On the other hand, since $\alpha_{i_k}(y_k)$ is in $\{\delta_i(y_k)\}$, up to a subsequence, $\{\frac{\delta(y_k)}{\E_o(\delta(y_k))}\}$ has limit in $\PMF(S)$ the same as $\frac{\alpha_{i_k}(y_k)}{\E_o(\alpha_{i_{k}}(y_k))}$. Hence $\{\frac{\delta(y_k)}{\E_o(\delta(y_k))}\}$ converges to $\eta$. So we complete the proof of Proposition \ref{proposition:pantscurveapp}. 
\end{proof}

\subsection{Regularity at pants curves.}\label{sec:regularitypants}
We are now in a position to prove that the assumptions in Lemma \ref{lemma:approximation} and Lemma \ref{lemma:log} hold. \\

\noindent {\bf More conventions:}~~~ From now on, we will use the hyperbolic length function $\ell_o(\cdot)$. Since $\ell_o^2(\cdot) \sim_{o} \E_{o}(\cdot)$, we can use $\ell_o(\cdot)$ to replace $\E_o(\cdot)$ without affecting the result when we defining the measure $\nu_o$, the embedding $\tau : \PMF(S) \longrightarrow \MF(S)$ and $\xi_{k}(\gamma)$. For instance, for a measurable subset $U \in \PMF(S)$, we have $$\nu_0(U) = \mu(\{\eta: [\eta] \in U, \ell_{o}(\eta) \leq 1 \}).$$  \\

\noindent {\bf Set-up 0}: Let $\alpha = \{\alpha_1, \cdots, \alpha_{3g-3}\}$ be a pants decomposition of $S$ and consider it to be a measured foliation still denoted by $\alpha$. Then $[\alpha]$ defines a unit holomorphic quadratic differential $q$ on $o$, namely the unique $q$ such that $[\mathcal{V}(q)] = [\alpha]$. Let $\xi = \frac{\alpha}{\ell_{o}(\alpha)}$, then $\xi$ is the image of $[\alpha]$ under $\tau$. We denote $g_t$ the Teichm\"{u}ller geodesic defined by $q$. We assume $\alpha$ is a thick pants curve, namely, for all $i \in \{1,\cdots,3g-3\}$, both $\ell_{o}(\alpha_i)$ and $\ell_{o}(\alpha)$ is bounded bleow and above, up to multiplicative constants depending only on $g,o,\epsilon$, by $e^{T}$ for some $T$ (cf. Proposition \ref{proposition:pantscurveapp}).\\

\begin{theorem}\label{theorem:regularity}
	Under the above {\bf Set-up 0}, there exist $M_0 > 0, C > 0$ and $D > 0$, depending on $g,o,\epsilon$ such that when $M < M_0$, we have $$C M^{\frac{h}{2}} \leq \nu(\{\eta \in \PMF(S): i(\eta, \xi) \leq M\}) \leq D M^{\frac{h}{2}}.$$
\end{theorem}
The main tool to prove the above theorem is the following Dehn-Thurston theorem.  Let $P = \{\alpha_k\}$ be a pants decomposition. For each $\alpha_k$, let $m_k: \MF(S) \longrightarrow \mathbb{R}_{\geq 0}, \xi \mapsto i(\alpha_k,\xi)$ be the intersection function defined by $\alpha_k$ and $t_k =tw_{k}$ be the twist function associated to $\alpha_k$.
\begin{theorem}[The Dehn-Thurston theorem \cite{PennerHarer},Theorem 3.1.1]\label{theorem:Dehn-Thurston}
	Let $S = S_g$ and $\alpha =\{\alpha_1, \cdots, \alpha_{3g-3}\}$ be a pants decomposition of $S$. Then the map
	\begin{equation}
	\begin{aligned}
	& \varpi: \MF(S) \longrightarrow \mathbb{R}^{6g-6}\\
	&\mathcal{F} \mapsto (m_1(\mathcal{F}),\cdots m_{3g-3}(\mathcal{F}), t_1(\mathcal{F}),\cdots, t_{3g-3}(\mathcal{F})).
	\end{aligned}
	\end{equation}
	gives a global coordinate for $\MF(S)$.	
\end{theorem}

There are various (equivalent) definitions for the Thurston measure $\nu_{Th}$ on $\MF(S)$. For doing computation later, the Dehn-Thurston theorem then enable us to define it to be the measure $\frac{1}{n!} \omega^{n}$ induced by the symplectic form $\omega = dm_1 \wedge dt_1 + \cdots + dm_{3g-3} \wedge dt_{3g-3}$. Notice that, although the symplectic form $\omega$ depends on $\alpha$, the meausre does not since $\MF(S)$ has piecewise integral linear structure \cite[Section 3.1]{PennerHarer} which means different pants decompositions give the same measure. Hence the constants does not depend on $\xi$ in Theorem \ref{theorem:regularity}. We now denote $\mu = \nu_{Th}$ with $\alpha$ is fixed to the one given in {\bf Set-up 0}. Note that $\frac{h}{2} = 3g-3$. We now prove the theorem.
\begin{proof}[Proof of Theorem \ref{theorem:regularity}]
	By Lemma \ref{Minskyinequailty}, for every two elements $\xi$ and $\eta$ in $\PMF(S)$, the intersection number $i(\eta,\xi) \leq 1$ and $1$ is achievable.  The proof is then divided into two parts. In the sequel, denote $a = \frac{1}{\sum_i\ell_{o}(\alpha_i)} $ and $\ell$ to be $$ \ell = \frac{1}{a\prod\left(\ell_{o}(\alpha_i)\right)^{\frac{2}{h}}} = \frac{\sum_i\ell_{o}(\alpha_i)}{\prod\left(\ell_{o}(\alpha_i)\right)^{\frac{2}{h}}}.$$ Then by our assumption (or cf. $(\star_2)$ in Proposition \ref{proposition:pantscurveapp} ), there exists $A_1 > 0$ and $B_1 > 0$ depending on $g, o, \epsilon$, such that $B_1 \leq \ell \leq A_1$.\\
	\noindent {\bf Upper bound:}~~~Namely, $\nu(\{\eta \in \PMF(S): i(\eta, \xi) \leq M\}) \leq D M^{\frac{h}{2}}$.
	
	Let $M \leq M_0$ and $M_0 = \frac{1}{4}$. By the definition of $\nu$ and $i(\eta,\xi) = a\sum^{3g-3}_{k=1}i(\eta,\alpha_k)$, we have,
	\begin{equation}\label{equation:regularityup1}
	\begin{aligned}
	&\nu\left(\left\{\eta \in \PMF(S): i(\eta, \xi) \leq M\right\}\right) \\
	&= \mu\left(\left\{t\eta \in \MF(S): i(\eta, \xi) \leq M, \ell_{o}(\eta) = 1, 0 \le t \le 1\right\}\right) (\mbox{by definition})\\
	&=\mu\left(\left\{t\eta \in \MF(S): a\sum_k m_k(\eta) \leq M, \ell_{o}(\eta) = 1,  0 \le t \le 1\right\}\right)\\
	&\leq \mu\left(\left\{ \eta: \forall k, m_k(\eta) \leq \frac{M}{a}, t_i(\eta)\ell_o(\alpha_i) \leq A_2 \right\}\right),
	\end{aligned}
	\end{equation}
	
	\noindent where $A_2$ is a constant depending only on $o$.  The last step comes from the fact that large twists will make the length to be large (cf. \cite[Lemma 3.2]{Kerckhoff_Nielsen} for a explicit formula for computing hyperbolic length). Thus there exists $D$, depends only on $g,o,\epsilon$, such that 
	\begin{equation}
	\begin{aligned}
	&\nu\left(\left\{\eta \in \PMF(S): i(\eta, \xi) \leq M\right\}\right)\\
	&\leq \mu\left( \left\{(m_1,\cdots, m_{3g-3},t_1,\cdots,t_{3g-3}): \forall k, m_k \leq \frac{M}{a}, t_k \leq \frac{A_2}{\ell_{o}(\alpha_k)}\right\}\right) (\mbox{by (\ref{equation:regularityup1})})\\
	&\leq A_3 M^{3g-3}\frac{1}{a^{3g-3}\prod^{3g-3}_{k=1}\ell_{0}(\alpha_k)}(\mbox{by the definition of $\mu$})\\
	&= A_3 M^{3g-3}\ell^{3g-3}\\
	&\leq A_3M^{3g-3}A_1^{3g-3} (\mbox{since $B_1 \leq \ell \leq A_1$})\\
	&\leq DM^{\frac{h}{2}}.
	\end{aligned}
	\end{equation}
 
	\noindent {\bf Lower bound:} ~~~ Namely, $C M^{\frac{h}{2}} \leq \nu(\{\eta \in \PMF(S): i(\eta, \xi) \leq M\})$.\\
 
	In order to bound the measure from below, we will construct a subset contained in the set. First fix, for each $i$, a positive orientation for $\alpha_i$. Let  
	\begin{equation}
	\begin{aligned}
	&V =\left\{t\eta \in \MF(S): i(\eta, \xi) \leq M, \ell_{o}(\eta) = 1, 0 \le t \le 1\right\}.
	\end{aligned}
	\end{equation}
	Then $\eta^0 = \frac{1}{3g-3}\xi$ is in $V$. Let $a$ and $M$ as above, and $\delta > 0$ be a positive number. Let $\kappa(g)$ be a sufficiently small positive number depending only on $o$. Define a set of $6g-6$-tuples by 
	\begin{displaymath}
	\begin{aligned}
	&W_0(a, M,\delta) = \\
	&\{ (x_1,\cdots,x_{3g-3},y_1,\cdots,y_{3g-3}): \forall i, 0 \leq ax_i \leq \kappa(g)M, 0 \leq y_i\ell_o(\alpha_i) \leq \delta\}.
	\end{aligned}
	\end{displaymath}
	Let also $\varpi$ be the coordinate map in Theorem \ref{theorem:Dehn-Thurston}. Then, on the one hand, by hyperbolic geometry (cf. \cite[Lemma 3.2]{Kerckhoff_Nielsen}), there is $M_0 > 0$ and $\delta_0 = \delta_0(M_0)$, depending on $o$, such that for all $M \leq M_0$, one has 
	\begin{displaymath}
	\begin{aligned}
	&\varpi^{-1}(\varpi(\eta^0) + W_0(a,M,\delta_0)) \subset V.
	\end{aligned}
	\end{displaymath}
	
	Notice that one could choose $a$ large enough such that $\varpi^{-1}$ is a homeomorphism on $\varpi(\eta^0) + W_0(a,M,\delta_0)$. Therefore $\mu(V) \geq \nu(\varpi^{-1}(\varpi(\eta^0) + W_0(a,M,\delta_0)))$. On the other hand, 
	\begin{displaymath}
	\begin{aligned}
	& \nu(\varpi^{-1}(\varpi(\eta^0) + W_0(a,M,\delta_0))) \\
	&= \nu( \varpi^{-1}(W_0(a,M,\delta_0))).
	\end{aligned}
	\end{displaymath}
	The last one is easy to see to be at least $CM^{\frac{h}{2}}$, where $C$ depends on $M_0$ and $o$. Hence the proof is complete. 
\end{proof}
\begin{proof}[Proof of Theorem \ref{Harish-Chandra}]
	Let $\gamma \in E_n$ and take $M = e^{-2n}$ in Corollary \ref{corollary:upperbound} and $\bar{M} = F\ln n e^{-2n}$ in Corollary \ref{corollary:lowerbound} where $F$ is the constant in Lemma \ref{lemma:inverseminsky}. Then Theorem \ref{theorem:regularity} and Proposition \ref{proposition:pantscurveapp} imply the assumption in  Lemma \ref{lemma:log} holds and hence $\Psi(\gamma)_{\geq M } \sim_{g,o,\epsilon} n$ and $\Psi(\gamma)_{\geq \bar{M} } \sim_{g,o,\epsilon} a_1n-c_1\ln\ln n$. Then by Corollary \ref{corollary:upperbound}, Corollary \ref{corollary:lowerbound} and Lemma \ref{lemma:approximation}, the proof is complete.
\end{proof}
In fact, we can prove the following theorem which will be used in showing uniform boundedness. Let $\zeta \in U(\theta,\epsilon)$. Define, for $N$ sufficiently small, $$\Psi(\zeta)_{ \geq N} =\int_{\{\eta \in \PMF(S): i(\eta,\zeta) > N\}}\left(\frac{1}{i(\eta, \zeta)}\right)^{\frac{h}{2}}d\nu(\eta).$$
\begin{theorem}\label{theorem:psiestimation}
    There exist $N_0$, $A> 0$ and $B> 0$, depending only on $g,\theta,o$, such that $$\Psi(\zeta)_{ \geq N} \leq -A\ln N +B.$$
\end{theorem}
\begin{proof}
The proof is clear by combining the last part of the proof of Lemma \ref{lemma:log}, as well as Lemma \ref{lemma:approximation}, Proposition \ref{proposition:pantscurveapp} and Theorem \ref{theorem:regularity}.
\end{proof}

\section{Ergodicity of boundary representation}
\subsection{Main theorem.}
 In this section, we will prove Theorem \ref{theoreom:ergodicitymcg}, namely,
\begin{theorem}\label{theorem:Ergdicitymainproof}
	Let $S = S_g (g \geq 2)$ and $\pi_{\nu}$ be the associated quasi-regular representation of the mapping class group $\M(S)$ on $L^2(\PMF(S),\nu)$. Let $n \gg \rho$ and $E_n = \mathcal{E}(\theta,\epsilon, n, o, \rho)$ (up to a subsequence). Let $e_n = Pr: E_n \longrightarrow \PMF(S)$ be the radial projection which assigns $g \in E_n$ to the direction $\xi_g$ of the oriented geodesic $[o, g \cdot o]$. Then the quasi-regular representation $\pi_{\nu}$ is ergodic with respect to $(E_n, e_n)$ and any $f \in \bar{H}^{L^{\infty}(\PMF(S),\nu)}$, where $$H = <\mathds{1}_{U}:\nu(\partial U) = 0 ~\mbox{and}~ U ~\mbox{is a Borel subset of}~ \PMF(S)>.$$
\end{theorem}
\begin{proof}
	The proof consists of verifying all assumptions in Theorem \ref{criterion2} for $E_n$. The first two will be verified by showing $E_n$ is of exponential growth (namely, Corollary \ref{corollary:exponentialgrowth} and Corollary \ref{weakconvergence}). The third one is by Proposition \ref{nontangentconvergence}. The last one is Theorem \ref{uniformbounded}.
\end{proof}
\begin{proposition}\label{nontangentconvergence}
	For every $n \gg \rho$, there are two sequences of real numbers $\{h_{r_n}(n, \rho)\}$ and $\{r_n\}$ such that $\lim_{n \rightarrow \infty}h_{r_n}(n,\rho) = \lim_{n\rightarrow \infty}r_n = 0$ and such that $$\forall n \in \mathbb{N}, \forall \gamma \in E_n, \frac{\langle \pi_{\nu}(\gamma)\mathds{1}_{\PMF(S)}, \mathds{1}_{\{x\in \PMF(S): i(x,e_n(\gamma))\geq r_n\}}\rangle}{\Phi(\gamma)} \leq h_{r_n}(n,\rho).$$
\end{proposition}
\begin{proof}
Now,Let $n \gg \rho$ and $\gamma \in E_n$. Let $\xi_{\gamma}$ as before. Take $r_n = \frac{1}{n}$, by the Harish-Chandra estimates (Theorem \ref{Harish-Chandra}), Corollary \ref{corollary:upperbound}, Lemma \ref{lemma:log} and the proof of its assumption (Theorem \ref{theorem:regularity}), $$\frac{\langle \pi_{\nu}(\gamma)\mathds{1}_{\PMF(S)}, \mathds{1}_{\{x\in \PMF(S): i(x,\xi_{\gamma})\geq \frac{1}{n}\}}\rangle}{\Phi(\gamma)} \leq c(g,o,\rho)\frac{\ln n-D}{a_1n-c_1\ln\ln n +b_1}.$$ Take $h(n,\rho) = c(g,o,\rho)\frac{\ln n - D}{a_1n-c_1\ln \ln n +b_1}$, we complete the proof.
\end{proof}

\subsection{Uniform boundedness.}\label{subsection:uniformboundedness}
In this section, we complete our proof of the main theorem by proving the uniform boundedness. We start with some lemmas comparing two types of neighborhoods.
\begin{lemma}\label{lemma:exponentialapproximate}
	Using notations as Lemma \ref{lemma:boundeddistance}. Let $\xi_{\gamma} \in \PMF(S)$ be the direction of $[o,y=\gamma \cdot o]$ and $\xi^{\gamma} \in \PMF(S)$ be the direction of $[o,x_{\gamma}=t_y]$. Then there exists a constant $C$ such that $i(\xi_{\gamma}, \xi^{\gamma}) \leq C e^{-2d(\gamma \cdot o,o)}$.
\end{lemma}
\begin{proof}
	By Lemma \ref{Minskyinequailty}, $$ i^2(\xi_{\gamma}, \xi^{\gamma}) \leq \E_{x_{\gamma}}(\xi_{\gamma})\E_{x_{\gamma}}(\xi^{\gamma}).$$
	Let $\alpha = \beta_y$ as before. Then $\xi^{\gamma} = \frac{\alpha}{\sqrt{\E_{o}(\alpha)}} \sim_{g,o} e^{-d(\gamma \cdot o,o)}\alpha$. As $\E_{x_{\gamma}}(\alpha) \sim_{g,o} 1$, we have  $\E_{x_{\gamma}}(\xi^{\gamma}) = \frac{1}{\E_{o}(\alpha)}\E_{x_{\gamma}}(\alpha) \prec_{g,o} e^{-2d(\gamma \cdot o,o)}$. On the other hand, by Lemma \ref{lemma:boundeddistance}, up to a multiplicative constant, one could replace $\E_{x_{\gamma}}(\xi_{\gamma})$ by $\E_{\gamma \cdot o}(\xi_{\gamma}) = e^{-2d(\gamma \cdot o,o)}$. Collect all discussions together, one can complete the proof.
\end{proof}

Let $\xi \in \PMF(S)$ and $x \in [o,\xi]$, denote $\mathcal{I}_C(\xi,x) = \{\eta \in \PMF(S): i(\eta,\xi) \leq C e^{-2d(x,o)} \}$. Let $L, \theta, \epsilon$ as in Theorem \ref{theorem:DDM}. 

\begin{lemma}\label{lemma:twoneigborhood}
	Let $\eta \in \PMF(S)$ and $\gamma \in E_n$. Suppose that $\eta$ does not leave $\T_{\epsilon}(S)$ eventually. Let $x \in [o,\eta]$ such that $d(x,o) = n$. Let $C \geq 1$ and $n$ large enough. Then if $\gamma \in E_n$ such that $i(\xi_{\gamma}, \eta) \leq Ce^{-2n}$, then $d(x, \gamma \cdot o) \leq \frac{1}{h}\ln\ln n$.   
\end{lemma} 
\begin{proof}
	We argue as Lemma \ref{lemma:inverseminsky}. Denote $\xi = \xi_{\gamma}$, hence by assumption $i(\xi,\eta) \leq Ce^{-2n}$. First we remark that, since both $\eta$ and $\xi$ are uniquely ergodic, we have a geodesic triangle $\triangle(o,\xi, \eta)$. As $\gamma \in E_n$, there is also a geodesic segment $I$ of length $\ell = \frac{1}{3h}\ln \ln n$ in $[o,\gamma \cdot o]$ ending at $p = \gamma \cdot o$ that has at least proportion $\theta$ in $\T_{\epsilon}(S)$. By Theorem \ref{theorem:DDM}, $$I \cap \mathcal{N}_{D}([o,\eta]\cap[\xi,\eta]) \ne \emptyset,$$ where $D$ is the constant given in Theorem \ref{theorem:DDM}. Choose $q \in I \cap \mathcal{N}_{D}([o,\eta]\cap[\xi,\eta])$. Then there are two possibilities:\\
	
	\noindent Case 1: $d(q, y) \leq D$ with $y\in [o,\eta]$.\\
	
	Then $$ d(q,o) - D \leq d(o,y) \leq d(q,o) +D.$$ Since $$n-\ell-\rho \leq d(q,o) \leq n+\rho,$$ we have $$0 \leq d(x,y) \leq \ell +D+\rho.$$ Hence,
	\begin{displaymath}
	\begin{aligned}
	&d(x,\gamma \cdot o) \leq d(x,y) + d(y,q) + d(q,p)\\
	&\leq \ell + D + D + \ell+\rho\\
	&\leq 2(\ell +D+\rho)\\
	&\leq 3 \ell. 
	\end{aligned}
	\end{displaymath} 
	
	\noindent Case 2: $d(q,y) \leq D$ with $y \in [\xi,\eta]$.
	
	Then by Lemma \ref{Minskyinequailty}, one has $$i^2(\eta,\xi) = \E_y(\xi)\E_{y}(\eta).$$ Now, since $d(q,y) \leq D$, by Kerckhoff's formula, we have $$e^{-2D}\E_{q}(\xi) \leq \E_{y}(\xi), ~~~e^{-2D}\E_{q}(\eta) \leq \E_{y}(\eta).$$ Therefore, $$ e^{-4D }\E_q(\xi)\E_{q}(\eta) \leq i^2(\xi,\eta).$$ On the other hand, we have $$\E_q(\xi) = e^{-2d(o,q)},~~~ i(\xi,\eta) \leq Ce^{-2n},$$ which implies that $$e^{-4D} \E_{q}(\eta)e^{-2d(o,q)} \leq C^2e^{-4n}.$$ 
	That is,  $$e^{-4D} \E_{q}(\eta)e^{2(n-d(o,q))} \leq C^2e^{-2n}.$$ By Kerckhoff's formula again, $$ \E_p(\eta) \leq C^2e^{2\rho+4D}e^{-2n},$$ i.e. $$ \frac{1}{2}\ln\E_p(\eta) \leq \ln (Ce^{\rho+2D})-n.$$ Apply Lemma \ref{busemann},  one could choose $z \in [o,\eta] \cap \T_{\epsilon}(S)$ so that, if denote $d(o,p) = t$, $d(p,z)=a$ and $d(z,o) = b$, then $a-b \leq -n + \ln (Ce^{2\rho+4D}) +1$. Therefore, we have $$ 0 \leq t + a-b \leq \ln (Ce^{2\rho+4D}) +1+ \rho =c_1.$$ Note that $c_1$ is a constant depending on $C, \epsilon,\theta$. Now consider the geodesic triangle $\Delta(o,p,z)$. Since the side $[o,p]$ has a segment $I =[s,p]$ ending at $p$ which has at least proportion $\theta$ in the thick part, take the midpoint $m$ of $I$, then since $\theta = 0.999$, the subsegment $[s,m]$ has at least $0.1$ in the thick part. By Theorem \ref{theorem:DDM} again, there exist $g \in [s,m]$ and a constant $D'$ when $n>>0$, such that, $$d(g, [o,z] \cup [p,z]) \leq D'.$$ If there is $h_1 \in [o,z]$ such that $d(h_1,g) \leq D'$, one can complete the proof as in the first case. Otherwise, there is a $h_2 \in [p,z]$ such that $d(h_2,g) \leq D'$. Then 
 \begin{equation*}
     \begin{aligned}
t+a-b = d(o,g) + d(g,p) + d(p,h_2) +d(h_2,z) - d(o,z)\\
= d(o,g) + D' + d(h_2,z) - d(o,z)+ d(g,p) + d(p,h_2)  -D'\\
\geq \ell - 2D' \to \infty,
\end{aligned}
 \end{equation*}
which is impossible since $t+a-b$ is bounded by $c_1$.
\end{proof}
\begin{corollary}\label{corollary:boundednumberpoints}
	Let $\eta \in \PMF(S)$ and suppose that $\eta$ does not leave $\T_{\epsilon}$ eventually. Let $x \in [o,\eta]$ such that $d(x,o) = n$. Let further $C > 0$ and $n$ large enough. Then 
	\begin{displaymath}
	\begin{aligned}
	&\left|\{ \gamma \in E_n: \gamma \cdot o \in Sec_{\mathcal{I}_C(\eta, x)}\}\right|\\
	& \prec_{g,o,\rho} n.
	\end{aligned}
	\end{displaymath}	
\end{corollary}
\begin{proof}
	By Theorem 1.2 in \cite{ABEM}, when $n$ is large enough, there exists a constant $N_0 > 0$, such that $|B(x, R) \cap \M(S)\cdot o| \leq N_0 e^{hR}$. Apply Lemma \ref{lemma:twoneigborhood}, we have the conclusion. 
\end{proof}

\begin{theorem}\label{uniformbounded}
	Under the notations used in Theorem \ref{theorem:Ergdicitymainproof}, we have
	$$\sup_{n}\left\|M^{\mathds{1}_{\PMF(S)}}_{E_n}\mathds{1}_{\PMF(S)}\right\|_{L^{\infty}(\PMF(S),\nu)} < \infty.$$
\end{theorem}
Recall that
\begin{displaymath}
\begin{aligned}
& M^{\mathds{1}_{\PMF(S)}}_{E_n}\mathds{1}_{\PMF(S)}([\xi])\\
&\phantom{=\;\;} = \frac{1}{|E_n|}\sum_{\gamma \in E_n}\frac{\pi_{\nu}(\gamma)\mathds{1}_{\PMF(S)}([\xi])}{\Phi(\gamma)}\\
&\phantom{=\;\;} = \frac{1}{|E_n|}\sum_{\gamma \in E_n}\left(\frac{\E_{o}(\xi)}{\E_{\gamma.o}(\xi)}\right)^{\frac{h}{4}}\frac{1}{\Phi(\gamma)}.
\end{aligned}
\end{displaymath}
By using the embedding map $\tau$ of $\PMF(S)$ into $\MF(S)$. One can rewrite the above formula to be
\begin{displaymath}
\begin{aligned}
& M^{\mathds{1}_{\PMF(S)}}_{E_n}\mathds{1}_{\PMF(S)}([\xi])\\
&\phantom{=\;\;} = \frac{1}{|E_n|}\sum_{\gamma \in E_n}\left(\frac{1}{\E_{\gamma.o}(\xi)}\right)^{\frac{h}{4}}\frac{1}{\Phi(\gamma)}.
\end{aligned}
\end{displaymath}
We first introduce a type of open sets $\mathcal{IN}$ in $\PMF(S)$ defined by intersection numbers. For every $\eta \in \PMF(S), C > 0, t > 0$, define
\begin{displaymath}
\begin{aligned}
& \mathcal{IN}(\eta,t,C) = \{\xi \in \PMF(S): i(\xi,\eta) \leq Ce^{-2t} \}.
\end{aligned}
\end{displaymath}

\begin{proof}[Proof of Theorem \ref{uniformbounded}]
Fix $R > R_0$ where $R_0$ is the constant in Lemma \ref{lemma:Shodowlemma}. Let $\epsilon$ sufficiently small so that for all $g \in \M(S)$, the open ball $B(g \cdot o, R) \subset \T_{\epsilon}(S)$ and Theorem \ref{theorem:DowDucMas} holds. We know that $U(\epsilon, \theta)$ is a subset of $\PMF(S)$ of full measure. We shall give a bound independent on $n \gg \rho$ for $ M^{\mathds{1}_{\PMF(S)}}_{E_n}\mathds{1}_{\PMF(S)}(\zeta)$ for every point $\zeta \in U(\epsilon, \theta)$.  As usual, for $\gamma \in E_n$, denote $\xi_{\gamma}$ to be the direction corresponding to $[o,\gamma \cdot o]$, hence a point in $\PMF(S)$. For each point $\gamma \cdot o$, consider the open ball $B(\gamma, R)$ of radius $R$ at $\gamma \cdot o$. Denote the projection of $B(\gamma, R)$ to $\PMF(S)$ by $\mathcal{O}(\gamma \cdot o,R)$. Then by  Lemma \ref{lemma:Shodowlemma}, the measure $\nu(\mathcal{O}(\gamma \cdot o,R)) \sim_{g,R,\rho} e^{-hn}$.
	Fix any $C \geq 1$, for instance $C = 1$. Divide $E_n$ to be two sets $E^1_n$ and $E^2_n = E_n-E^1_n$ where $E^1_n$ consists of $\gamma \in E_n$ so that $\xi_{\gamma} \notin \mathcal{IN}(\zeta, n, C) $. We then have, for each $\zeta \in U(\epsilon,\theta)$,
	\begin{equation} \label{equation:uniformboundedness}
	\begin{aligned}
	& M^{\mathds{1}_{\PMF(S)}}_{E_n}\mathds{1}_{\PMF(S)}(\zeta)\\
	&\phantom{=\;\;} = \frac{1}{|E_n|}\sum_{\gamma \in E_n}\left(\frac{1}{\E_{\gamma.o}(\zeta)}\right)^{\frac{h}{4}}\frac{1}{\Phi(\gamma)}\\
	&\phantom{=\;\;} = \underbrace{\frac{1}{|E_n|}\sum_{\gamma \in E^1_n}\left(\frac{1}{\E_{\gamma.o}(\zeta)}\right)^{\frac{h}{4}}\frac{1}{\Phi(\gamma)}}_{I} + \underbrace{\frac{1}{|E_n|}\sum_{\gamma \in E^2_n}\left(\frac{1}{\E_{\gamma.o}(\zeta)}\right)^{\frac{h}{4}}\frac{1}{\Phi(\gamma)}}_{II}\\
	&= {\rm I + II}.
	\end{aligned}
	\end{equation}
	
	First we want to bound term ${\rm I}$ in Equation (\ref{equation:uniformboundedness}). The set $E^1_n$ can be further decomposed into two sets: $F^1_n$ and $F^2_n = E^1_n - F^1_n$, where $F^1_n = \{ \gamma \in E^1_n: \mathcal{O}(\gamma \cdot o,R) \cap \mathcal{IN}(\zeta,n,C) = \emptyset\}.$ One then has,
	\begin{equation}
	\begin{aligned}
	&{\rm I} = \underbrace{\frac{1}{|E_n|}\sum_{\gamma \in F^1_n}\left(\frac{1}{\E_{\gamma.o}(\zeta)}\right)^{\frac{h}{4}}\frac{1}{\Phi(\gamma)}}_{III} + \underbrace{ \frac{1}{|E_n|}\sum_{\gamma \in F^2_n}\left(\frac{1}{\E_{\gamma.o}(\zeta)}\right)^{\frac{h}{4}}\frac{1}{\Phi(\gamma)}}_{IV}\\
	&\phantom{=\;\;} = {\rm III + IV}.
	\end{aligned}
	\end{equation}

	For the term ${\rm III}$. First notice that $$\forall y \in B(\gamma \cdot o,R), \frac{1}{\E_{\gamma \cdot o}(\zeta)} \sim_{R} \frac{1}{\E_{y}(\zeta)}.$$ Now, by Lemma \ref{Minskyinequailty}, for $\nu-$almost every $\xi_y \in \mathcal{O}(\gamma \cdot o, R)$,$$ \left(\frac{1}{\E_y(\zeta)}\right)^{\frac{h}{4}} \prec_{\rho,R} \frac{1}{e^{\frac{hn}{2}} (i(\xi_y, \zeta))^{\frac{h}{2}}}.$$ Hence, for $\nu-$almost every $ \xi \in \mathcal{O}(\gamma \cdot o,R),$
	\begin{displaymath}
	\begin{aligned}
	&\left(\frac{1}{\E_{\gamma \cdot o}(\zeta)}\right)^{\frac{h}{4}} \prec_{\rho,R} e^{-\frac{hn}{2}} \frac{1}{(i(\xi, \zeta))^{\frac{h}{2}}}.
	\end{aligned}
	\end{displaymath}
	Therefore, 
	\begin{displaymath}
	\begin{aligned}
	& {\rm III} = \frac{1}{|E_n|}\sum_{\gamma \in F^1_n}\left(\frac{1}{\E_{\gamma.o}(\zeta)}\right)^{\frac{h}{4}}\frac{1}{\Phi(\gamma)}\\
	& \prec_{R} \frac{1}{|E_n|}\sum_{\gamma \in F^1_n} \frac{e^{-\frac{hn}{2}}}{\nu(\mathcal{O}(\gamma \cdot o, R))} \int_{\mathcal{O}(\gamma \cdot o, R)}\frac{1}{(i(\eta, \zeta))^{\frac{h}{2}}}d\nu(\eta) \frac{1}{\Phi(\gamma)}.
	\end{aligned}
	\end{displaymath}
	
	Note that there are bounded number of intersections of open sets of the form $\mathcal{O}(\gamma \cdot o,R)$ and the bound depends on $R$ and $\rho$. Thus, since $|E_n| \asymp e^{hn}$ (Corollary \ref{corollary:exponentialgrowth}) and $\Phi(\gamma) \succ_{g,o,\rho} (a_1n-c_1\ln \ln n +b_1)e^{-\frac{hn}{2}}$ (Harish-Chandra estimates), substitute all these together, one has,
	\begin{equation}
	\begin{aligned}
	&{\rm III} \prec_{g,o,\rho,R} \frac{1}{a_1n-c_1\ln \ln n +b_1}\int_{\{\eta \in \PMF(S): i(\eta,\zeta) > Ce^{-2n}\}}\left(\frac{1}{i(\eta, \zeta)}\right)^{\frac{h}{2}}d\nu(\eta)\\
	& \prec_{g,o,\rho,R} 1.
	\end{aligned}
	\end{equation}
	The last inequality follows from the fact that $\zeta \in U(\epsilon,\theta)$ and Theorem \ref{theorem:psiestimation}. \\
	
	For terms ${\rm IV}$ and ${\rm II} $. Take $H_n = E^2_n \cup F^2_n$. These two terms can be put together to obtain:
	\begin{equation}
	\begin{aligned}
	& {\rm IV + II} \\
	& = \frac{1}{|E_n|}\sum_{\gamma \in H_n}\left(\frac{1}{\E_{\gamma.o}(\zeta)}\right)^{\frac{h}{4}}\frac{1}{\Phi(\gamma)}\\
	&\leq \frac{1}{|E_n|}\sum_{\gamma \in H_n} \frac{e^{\frac{hL(\gamma)}{2}}}{\Phi(\gamma)}\\
	& \sim_{g,\rho,o} e^{-hn} \sum_{\gamma \in H_n}\frac{e^{\frac{hn}{2}}}{(a_1n-c_1\ln \ln n+b_1)e^{\frac{-hn}{2}}}\\
	& = \frac{1}{a_1n-c_1\ln\ln n+b_1}|H_n|.	
	\end{aligned}
	\end{equation}
	We now \textit{claim:}  $|H_n| \prec_{g,o,\rho,\epsilon,\theta,R} n$. Thus the sum ${\rm IV + II}$ tends to $0$ when $n \rightarrow \infty$ which complete the proof of the theorem.\\
	
	It remains to prove the above \textit{claim}. 
	\begin{proof}[Proof of the \textit{claim}]
		By Corollary \ref{corollary:boundednumberpoints}, the number $|E^2_n| \prec n$. We now show that so is $|F^2_n|$. By the choice of $R$ and $\epsilon$, for every $\gamma \in \M(S)$ and every point $q \in B(\gamma \cdot o, R)$, we have $q \in \T_{\epsilon}(S) $ and $d_T(\gamma \cdot o, q) \prec_{R} 1$. Assume now that $\gamma \in F^2_n$, namely $\mathcal{O}(\gamma \cdot o, R) \cap \mathcal{IN}(\zeta, n, C) \ne \emptyset$. As $U(\epsilon, \theta)$ has full measure, in particular, it is dense in $\PMF(S)$, thus one can choose $q \in B(\gamma \cdot o, R)$ so that the direction $\xi_q$ of $[o,q]$ is in $U(\epsilon, \theta) \cap \mathcal{IN}(\zeta, n, C)$. By Theorem \ref{theorem:fellowtravelling}, there is a $P = P(\epsilon,R)$, so that the two geodesics $[o,\gamma \cdot o]$ and $[o,q]$ are $P-$fellow travelling in a parametrized fashsion. Now consider the $P-$neighborhood $\mathcal{N}_P$ of $\T_{\epsilon}(S)$, namely the union of points in $\T(S)$ that has distance at most $P$ with a point in $\T_{\epsilon}(S)$. As $\M(S)$ acts as isometries on $\T(S)$ and $\T_{\epsilon}(S)$ is $\M(S)-$invariant and cocompact, the neighborhood $\mathcal{N}_P$ is $\M(S)-$invariant and cocompact. By Mumford's compactness, there is a small $\epsilon'$ so that $\mathcal{N}_P \subset \T_{\epsilon'}(S)$. Then as $\gamma \in E_n$, the geodesic segment $[o,q]$ has the property that it contains a segment $I= [a,q]$ of length $\frac{1}{3h}\ln \ln n$ such that $I$ has at least $\theta$ in $\T_{\epsilon'}(S)$. By Theorem \ref{theorem:DDM}, there are constants $D'= D'(\epsilon',\theta)$ and $L_0'= L_0'(\epsilon',\theta)$ satisfy Theorem \ref{theorem:DDM}. Take $n$ large enough and repeat the proof of Lemma \ref{lemma:twoneigborhood} and Corollary \ref{corollary:boundednumberpoints}, one has $|F^2_n| \prec n$.	
	\end{proof}
\end{proof}


\bibliography{unitary}
\addcontentsline{toc}{section}{References}
\bigskip
\noindent
Faculty of Mathematics, Technion-Israel Institute of Technology, Haifa, Israel
3200003\\
\noindent \textbf{Email address: biaoma@campus.technion.ac.il}
\end{document}